\documentclass[11pt,reqno]{amsart}
\usepackage[usenames]{color}
\usepackage{amsmath, amssymb, latexsym, multicol, amsthm, color, enumitem, tensor, comment, mathrsfs, tikz}
\usepackage[all]{xy}
\usepackage{cancel}
\usepackage[T1]{fontenc}
\usepackage{hyperref}
\usepackage{cleveref}
\newtheoremstyle{natural}
{\parsep}   
{\parsep}   
{\normalfont}  
{0pt}       
{\bfseries} 
{.}         
{5pt plus 1pt minus 1pt} 
{}          

\newtheoremstyle{remark}
{\parsep}   
{\parsep}   
{\normalfont}  
{0pt}       
{\itshape} 
{.}         
{5pt plus 1pt minus 1pt} 
{}          
\theoremstyle{plain}
\newtheorem{thm}{Theorem}[section]
\newtheorem{cor}[thm]{Corollary}
\newtheorem{lem}[thm]{Lemma}
\newtheorem{prp}[thm]{Proposition}

\theoremstyle{remark}

\newtheorem*{ntn}{Notation}
\numberwithin{equation}{section}

\crefname{thm}{theorem}{theorems}
\crefname{lem}{lemma}{lemmata}
\crefname{prp}{proposition}{propositions}
\crefname{cor}{corollary}{corollaries}

\newcommand{\Z}{\mathbb{Z}}
\newcommand{\Q}{\mathbb{Q}}
\newcommand{\N}{\mathbb{N}}
\newcommand{\R}{\mathbb{R}}
\newcommand{\C}{\mathbb{C}}

\newcommand{\G}{\mathbb{G}}

\newcommand{\bs}{\backslash}
\newcommand{\cb}[1]{\left\{{#1}\right\}}
\newcommand{\cbm}[2]{\left\{{#1}\;\middle|\;{#2}\right\}}

\newcommand{\diag}{\operatorname{diag}}

\newcommand{\id}{\operatorname{id}}

\newcommand{\pb}[1]{\left\langle{#1}\right\rangle}

\newcommand{\ppmod}[1]{\hspace{-0.3cm}\pmod{#1}}

\newcommand{\ol}[1]{\overline{#1}}

\newcommand{\rb}[1]{\left({#1}\right)}

\newcommand{\rsa}{\rightsquigarrow}
\newcommand{\sbe}{\subseteq}

\newcommand{\vb}[1]{\left| {#1} \right|}
\newcommand{\vol}{\operatorname{vol}}
\newcommand{\y}{\operatorname{y}}

\newcommand{\Ad}{\operatorname{Ad}}

\newcommand{\Kl}{\operatorname{Kl}}

\newcommand{\GL}{\operatorname{GL}}

\newcommand{\SO}{\operatorname{SO}}
\newcommand{\Sp}{\operatorname{Sp}}
\newcommand{\GSp}{\operatorname{GSp}}

\newcommand{\mc}{\mathcal}
\newcommand{\mf}{\mathfrak}
\newcommand{\ms}{\mathscr}

\newcommand{\bdd}{\begin{center}\begin{tikzcd}}
\newcommand{\bd}{\begin{tikzcd}}
\newcommand{\edd}{\end{tikzcd}\end{center}}
\newcommand{\ed}{\end{tikzcd}}
\newcommand{\bdp}{\begin{center}\begin{tikzpicture}}
\newcommand{\edp}{\end{tikzpicture}\end{center}}
\newcommand{\bi}{\begin{itemize}}
\newcommand{\ei}{\end{itemize}}
\newcommand{\bt}{\begin{tikzpicture}}
\newcommand{\et}{\end{tikzpicture}}
\newcommand{\ba}{\[\begin{aligned}}
\newcommand{\ea}{\end{aligned}\]}
\newcommand{\bp}{\begin{pmatrix}}
\newcommand{\ep}{\end{pmatrix}}
\newcommand{\bv}{\begin{vmatrix}}
\newcommand{\ev}{\end{vmatrix}}
\newcommand{\bb}{\begin{bmatrix}}
\newcommand{\eb}{\end{bmatrix}}
\newcommand{\bB}{\begin{Bmatrix}}
\newcommand{\eB}{\end{Bmatrix}}
\newcommand{\bea}{\begin{enumerate}[label=(\alph*)]}
\newcommand{\ber}{\begin{enumerate}[label=(\roman*)]}
\newcommand{\ben}{\begin{enumerate}[label=(\arabic*)]}
\newcommand{\ee}{\end{enumerate}}

\title{A Density Theorem for $\Sp(4)$}

\author{Siu Hang Man}
\address{Siu Hang Man, Mathematisches Institut, Universit\"at Bonn, Endenicher Allee 60, 53115 Bonn, Germany}
\email{shman@math.uni-bonn.de}

\thanks{The author is supported in part by DAAD Graduate School Scholarship Programme, and by the German Research Foundation under Germany's Excellence Strategy – EXC-2047/1 – 390685813.}

\begin{document}

\begin{abstract}
Strong bounds are obtained for the number of automorphic forms for the group $\Gamma_0(q) \sbe \Sp(4,\Z)$ violating the Ramanujan conjecture at any given unramified place, which go beyond Sarnak's density hypothesis. The proof is based on a relative trace formula of Kuznetsov type, and best-possible bounds for certain Kloosterman sums for $\Sp(4)$. 
\end{abstract}
\date{\today}
\makeatletter
\@namedef{subjclassname@2020}{%
  \textup{2020} Mathematics Subject Classification}
\makeatother
\maketitle

\section{Introduction}
In number theory, we often study families of objects, and this is especially the case in the context of automorphic forms, and this methodology has been described formally in \cite{SarnakLetterc, SST2016}. On one hand, this can smooth out irregularities of individual members (which may exist, or whose proof of non-existence is beyond our reach) and allows the use of statistical concepts and deformation techniques to investigate properties within a given family. On the other hand, strong analytic tools, such as various types of trace formulae, are available, to give non-trivial statements to the families as a whole. 

One of the key conjectures in the theory of automorphic forms is the Ramanujan conjecture, which states that cuspidal automorphic representations of the group $\GL(n)$ over a number field $F$ are tempered. While the conjecture is still far out of reach even for the $\GL(2)$ case, one can think of generalisations to automorphic forms of other reductive groups, such as $\Sp(2n)$. It is well-known that the naive generalisation of the Ramanujan conjecture is false for the group $\Sp(4)$, because of the presence of Saito-Kurokawa lifts, which are not tempered. This is not the end of the investigation, however, because it is still an open question whether generic cuspidal automorphic representations of $\Sp(4)$ are tempered. While this is still far out of reach, we can consider approximations to the conjecture as a substitute, and try to bound the number of members in a family violating the conjecture relative to the amount by which they violate the conjecture. This gives a density result, which is analogous to the ones given in \cite{Blomer2019b} for the $\GL(n)$ case. This does not prove the conjecture, but such density results often suffice in applications. 

In this paper we consider the family of generic cuspidal automorphic representations for the group $\Gamma_0(q) \sbe \Sp(4,\Z)$ of matrices whose lower left $2\times 2$ block is divisible by $q$. Fix a place $v$ of $\Q$. For an automorphic representation $\pi = \bigotimes\limits_v \pi_v$, we denote by $\mu_\pi(v) = (\mu_\pi(v,1), \mu_\pi(v,2))$ its local Langlands spectral parameter (see \eqref{eq:L_para_def}), each entry viewed modulo $\frac{2\pi i}{\log p}\Z$ if $v=p$ is a prime. We write 
\begin{align}\label{eq:VioPar}
\sigma_\pi (v) = \max\cb{\vb{\Re\mu_\pi(v,1)}, \vb{\Re\mu_\pi(v,2)}}.
\end{align}
The representation $\pi$ is tempered at $v$ if $\sigma_\pi(v) = 0$, and the size of $\sigma_\pi(v)$ gives a measure on how far $\pi$ is from being tempered at $v$. An example of a non-tempered representation is the trivial representation, which satisfies $\sigma_{\operatorname{triv}} (v) = 3/2$ for all places $v$.

For a finite family $\mc F$ of automorphic representations of $\Sp(4)$ and $\sigma\geq 0$ we define
\ba
N_v(\sigma, \mc F) = \vb{\cbm{\pi\in\mc F}{\sigma_\pi(v) \geq \sigma}}.
\ea
Trivially, we have $N_v(0,\mc F) = \vb{\mc F}$, and if $\mc F$ contains the trivial representation, then we have $N_v(3/2, \mc F) \geq 1$. One may hope to interpolate linearly between the two extreme cases, and obtain a bound of the form
\begin{align}\label{eq:bound_general_form}
N_v(\sigma, \mc F) \ll_{v, \varepsilon} \vb{\mc F}^{1-\frac{\sigma}{a}+\varepsilon}
\end{align}
with $a = 3/2$. In the context of groups $G$ of real rank 1, for the principal congruence subgroup $\Gamma(q) = \cbm{\gamma \in G(\Z)}{\gamma = \id\pmod{q}}$ and $v=\infty$, this is known as Sarnak's density hypothesis \cite[p. 465]{Sarnak1990}.

Density theorems have attracted much attention in the history, and many strong density results are known for various automorphic families on $\GL(2)$ with different settings \cite{Huxley1986, Sarnak1987, Iwaniec1990, BM1998, BM2003, BBR2014}. Via Kuznetsov-type trace formulae on $\GL(3)$, strong density results on $\GL(3)$ were obtained in \cite{Blomer2013, BBR2014, BBM2017}. Blomer \cite{Blomer2019b} further generalised the technique to obtain results in $\GL(n)$ beyond Sarnak's density hypothesis. However, relatively little is known for general reductive groups. Finis-Matz \cite{FM2019} gives as by-products some density results for the family of Maa\ss{} forms of Laplace eigenvalue up to a height $T$ and fixed level. However, the value of $a$ is large, and is at least quadratic in $r$, the rank of the group, so that even the ``convexity bound'' cannot be obtained.

More concretely, in this paper we consider the family $\mc F_I(q)$ of generic cuspidal automorphic representations for the group $\Gamma_0(q)\sbe \Sp(4,\Z)$ for a large prime $q$, and Laplace eigenvalue $\lambda$ in a fixed interval $I$. When the size of $I$ is sufficiently large, we have $\vb{\mc F_I(q)} \asymp_I q^3$. For this family and any place $v\neq q$ of $\Q$, we go beyond the density hypothesis and obtain $a=3/4$, which is halfway between the density hypothesis and the Ramanujan conjecture.

\begin{thm}\label{thm:density_thm}
Let $q$ be a prime, and $v$ a place of $\Q$ different from $q$, $I \sbe [0,\infty)$ a fixed interval, $\varepsilon>0$, and $\sigma\geq 0$. Then
\ba
N_v(\sigma, \mc F_I(q)) \ll_{I,v,n,\varepsilon} q^{3-4\sigma+\varepsilon}.
\ea
\end{thm}

The proof is based on a careful analysis on the arithmetic side of the Kuznetsov formula, and on the spectral side through a relation of Fourier coefficients of automorphic forms and Hecke eigenvalues. Let $\lambda (m, \pi)$ be the Hecke eigenvalue of $\pi\in\mc F_I(q)$ for the $m$-th standard Hecke operator $T(m)$. It is convenient to adopt the normalisation $\lambda'(m, \pi) := m^{-3/2} \lambda(m,\pi)$. 

\begin{thm}\label{thm:lambdaZ}
Keep the notations as above. Let $m\in\N$ be coprime to $q$ and $Z\geq 1$. Then
\ba
\sum\limits_{\pi\in\mc F_I(q)} \vb{\lambda'(m,\pi)}^2 Z^{2\sigma_\pi(\infty)} \ll_{I,\varepsilon} q^{3+\varepsilon}
\ea
uniformly in $mZ\ll q^2$ for a sufficiently small implied constant depending on $I$.
\end{thm} 

Let us roughly sketch the proof of \Cref{thm:lambdaZ}. We denote by $\cb{\varpi}$ an orthonormal basis of right $K$-invariant automorphic forms for $\Gamma_0(q)$, cuspidal or Eisenstein series, where $K$ is the maximal compact subgroup of $\Sp(4,\R)$. We denote by $\int_{(q)} d\varpi$ the integral over the complete spectrum of $L^2(\Gamma_0(q)\bs \Sp(4,\R)/K)$. Very roughly, the Kuznetsov formula takes the form
\begin{align}
\int_{(q)} \vb{A_\varpi(M)}^2 Z^{2\sigma_\pi(\infty)} \delta_{\lambda_\varpi \in I} d\varpi \; \text{``}\approx\text{''} \; 1 + \sum\limits_{\id\neq w \in W} \sum\limits_{c_1, c_2} \frac{\Kl_{q,w}(c,M,M)}{c_1c_2},
\end{align}
where $M = (1,m) \in \Z^2$, $A_\varpi(M)$ is the $M$-th Fourier coefficient of $\varpi$, defined in \eqref{eq:Fourier_in_Whittaker}, $W$ is the Weyl group of $\Sp(4)$, and $\Kl_{q,w}(c,M,M)$ is a generalised Kloosterman sum, defined in \eqref{eq:Kloosterman_definition} below, associated with the Weyl element $w$, and moduli $c = (c_1,c_2)$. Note that the Kuznetsov formula only extracts the generic spectrum.

However, the situation here is very different from $\GL(n)$ case found in \cite{Blomer2019b}. In the symplectic case, there are no simple relations between the Fourier coefficients $A_\varpi(M)$ of a cuspidal newform $\varpi$ and Hecke eigenvalues $\lambda'(m,\pi)$ of the corresponding automorphic representation (i.e. $\varpi \in V_\pi$). This is in stark contrast with the $\GL(n)$ case, where the Fourier coefficients and Hecke eigenvalues are proportional \cite[Theorem 9.3.11]{Goldfeld2006}. It is because of this obstacle that the Kuznetsov formula is not yet a standard tool for the group $\GSp(4)$, and the present paper seems to be the first application of the Kuznetsov formula that is seen in action for a group other than $\GL(n)$. 

While the Fourier coefficients in principle contain the information on Hecke eigenvalues, it is not obvious how to extract it. A detailed analysis of the relations between them is found in \Cref{section:Hecke_Fourier}. In \Cref{thm:eigenvalue_Fourier} we establish a recursive formula of $\lambda(p^r, \pi)$ in terms of Fourier coefficients. We also outline an algorithm for computing arbitrary Fourier coefficients of a cuspidal form in terms of its Hecke eigenvalues in the appendix. While this is not needed for the proof of the theorems, such results serve an independent interest in number theory, in laying the groundwork for further applications of the Kuznetsov formula on $\Sp(4)$, as well as Fourier analysis of automorphic forms on $\Sp(4)$ in general. 

Using \Cref{thm:eigenvalue_Fourier}, we deduce from \Cref{lem:Fourier_bound} that for a prime $p \nmid q$ and $r\in\N$, the size of Fourier coefficients $A_\varpi(1,p^r)$ of an $L^2$-normalised generic cuspidal form $\varpi$ is often as big as $q^{-3/2-\varepsilon} p^{r\sigma_\pi(p)}$. Through this relation, we are able to use the Kuznetsov formula to derive information on $\sigma_\pi(p)$ from an analysis of the Kloosterman sums. Meanwhile, the factor $Z^{2\sigma_\pi(\infty)}$ deals with the infinite place, so the test function $ \vb{A_\varpi(M)}^2 Z^{2\sigma_\pi(\infty)}$ treats the finite places and the infinite place essentially on the same footing.

When $mZ\ll q$, the Kloosterman sums associated to non-trivial Weyl elements are empty, hence the off-diagonal terms vanish completely. We will use this observation to prove \Cref{thm:alpham} below. To obtain stronger density results, we have to deal with the Kloosterman sums appearing in the off-diagonal term, and improve the trivial bound $\vb{S_{q,w}(c,M,N)} \leq c_1c_2$. Obtaining such bounds for general groups remains a major open problem. 

Finally, we give an application of \Cref{thm:lambdaZ}, for a large sieve inequality analogous to the $\GL(n)$ case \cite{Blomer2019b}. 

\begin{thm}\label{thm:alpham}
Let $q$ be prime and $\cb{\alpha(m)}_{m\in\N}$ any sequence of complex numbers. Then
\ba
\sum\limits_{\pi \in\mc F_I(q)} \Big|\sum\limits_{m\leq x} \alpha(m) \lambda'(m,\pi)\Big|^2 \ll_{I,\varepsilon} q^3 \sum\limits_{m\leq x} \vb{\alpha(m)}^2
\ea
uniformly in $x\ll q$ for a sufficiently small implied constant depending on $I$.
\end{thm}

And as a corollary, we establish a bound for the second moment of spinor L-functions on the critical line. Precisely, let $L(s,\pi)$ be the spinor L-function associated to $\pi$, normalised such that its critical strip is $0<\Re s<1$. 

\begin{cor}\label{cor:Lvalue}
For $q$ prime and $t\in\R$, we have
\ba
\sum\limits_{\pi \in\mc F_I(q)} \vb{L(1/2+it, \pi)}^2 \ll_{I,t,\varepsilon} q^{3+\varepsilon}. 
\ea
\end{cor}

\section*{Acknowledgement}
The author would like to thank Valentin Blomer for his guidance, and Edgar Assing for his helpful explanation on the subject.

\section{Preliminaries}
Let $U \sbe \Sp(4)$ be the standard unipotent subgroup
\ba
U &= \cbm{\bp 1 & x_{12} & x_{13} & x_{14}\\ & 1 & x_{23} & x_{24}\\ &&1\\&&x_{43}&1\ep}{\begin{array}{l} x_{ij} \in \G_a,\\ x_{12} = -x_{43},\\ x_{14} = x_{23}+x_{12}x_{24}\end{array}}.
\ea
Let $V\sbe \Sp(4)$ be the group of diagonal matrices with entries $\pm 1$. Let $W$ be the Weyl group of $\Sp(4)$, which is generated by the matrices
\ba
s_\alpha &= \bp &1\\-1\\&&&1\\&&-1 \ep, & s_\beta &= \bp 1\\ &&&1\\ &&1\\&-1\ep
\ea
as the elements in $\Sp(4)/V$. We denote the long element of the Weyl group by $w_0:= s_\alpha s_\beta s_\alpha s_\beta$. For $w\in W$, we define $U_w = w^{-1}U^\top w \cap U$. For $N \in \R^2$ we define a character $\theta_N: U(\R) \to S^1$ by
\begin{align}\label{eq:character_def}
\theta_N (x) = e\rb{N_1 x_{12} + N_2 x_{24}}.
\end{align}
Note that if $N \in \Z^2$, this defines a character $U(\R)/U(\Z) \to S^1$. If $N = (1,1)$, we drop it from the notation of the character.  

Let $T\sbe \Sp(4)$ be the diagonal torus. The standard minimal parabolic subgroup is given by $P_0 = TU$. We embed $y = (y_1, y_2) \in \R_+^2$ into $T(\R)$ via the map $\iota(y) = (y_1y_2^{1/2}, y_2^{1/2}, 1/y_1y_2^{1/2}, 1/y_2^{1/2})$. We denote the image of $\R_+^2$ in $T(\R)$ by $T(\R_+)$. An element $g\in \Sp(4,\R)$ admits Iwasawa decomposition $g = xyk$, with $x = U(\R)$, $y \in T(\R_+)$ and $k \in K$, where $K = \SO(4,\R) \cap \Sp(4,\R)$ is the maximal compact subgroup of $\Sp(4,\R)$. We denote by $\y(g) = \iota^{-1}(y)$ the Iwasawa $y$-coordinates of $g$. For $w\in W$, $y\in\R_+^2$ we write $\tensor[^w]{y}{} = \y(w \iota(y)^{-1} w^{-1})$. 

For $\alpha \in \C^2$, $y\in \R_+^2$, we write $y^\alpha = y_1^{\alpha_1} y_2^{\alpha_2}$. Let $\eta = (2, 3/2)$. We define measures
\ba
dx &= dx_{12} dx_{13} dx_{23} dx_{24}, &  d^*y = y^{-2\eta} \frac{dy_1}{y_1} \frac{dy_2}{y_2}
\ea
on $U(\R)$ and $\R_+^2$ respectively. We denote the pushforward of $d^*y$ to $T(\R_+)$ by $\iota$ also by $d^*y$. Then $dx$ is the Haar measure on $U(\R)$, and $dx d^*y$ is a left $\Sp(4,\R)$-invariant measure on $\Sp(4,\R)/K$. 

We define another embedding of $\R_+^2$ into $T(\R_+)$ by
\ba
c = (c_1, c_2) \mapsto c^* = \diag (1/c_1, c_1/c_2, c_1, c_2/c_1).
\ea
A simple calculation shows that $\y(c^*)^\eta = (c_1c_2)^{-1}$.

Let $\pi = \bigotimes \pi_v$ be a globally generic irreducible spherical representation of $\GSp(4)$ with trivial central character. Using notations in \cite{RS2007}, $\pi_v$ is induced from the character $\chi_1 \times \chi_2 \rtimes \sigma$, given by
\ba
\diag\rb{t_1,t_2,t_1^{-1}v, t_2^{-1}v} \mapsto \chi_1(t_1) \chi_2(t_2) \sigma(v). 
\ea
As $\pi_v$ is right $K_v$-invariant, we may assume that $\chi_1, \chi_2, \sigma$ are unramified, and we may write $\chi_1 = \vb{\;\cdot\;}^{\alpha_1}$, $\chi_2 = \vb{\;\cdot\;}^{\alpha_2}$ and $\sigma =  \vb{\;\cdot\;}^{\beta}$. As $\pi_v$ has trivial central character, we have $\alpha_1+\alpha_2+2\beta = 0$. So the L-parameter is given by
\ba
\rb{\chi_1\chi_2\sigma, \chi_1\sigma, \chi_2\sigma, \sigma} = \rb{\frac{\alpha_1+\alpha_2}{2}, \frac{\alpha_1-\alpha_2}{2}, \frac{-\alpha_1+\alpha_2}{2}, \frac{-\alpha_1-\alpha_2}{2}}.
\ea
We then take 
\begin{align}\label{eq:L_para_def}
\mu_\pi(v) = \rb{\frac{\alpha_1+\alpha_2}{2}, \frac{\alpha_1-\alpha_2}{2}},
\end{align}
so the L-parameter becomes $\rb{\mu_\pi(v,1), \mu_\pi(v,2), -\mu_\pi(v,1), -\mu_\pi(v,2)}$. When $\pi_v$ is lifted to a self-dual representation of $\GL(4)$, this is precisely the natural Langlands parameter of the lift.

\section{Auxiliary results}

Since the Iwasawa decomposition $\Sp(4,\R) = U(\R)T(\R_+)K$ is actually the Gram-Schmidt orthogonalisation of rows, we can compute $\y(g)$ explicitly. Let $\Delta_1$ be the norm of the third row of $g$, and $\Delta_2$ be the area of the parallelogram spanned by the bottom two rows of $g$. Then we have
\ba
g \equiv \bp 1/\Delta_1 & * & * & *\\ & \Delta_1/\Delta_2 & * & *\\ &&\Delta_1\\&& * & \Delta_2/\Delta_1\ep \pmod{K}. 
\ea
In particular, we have $\y(g) = (\Delta_2/\Delta_1^2, \Delta_1^2/\Delta_2^2)$. Conversely, if $\y(g) = (Y_1, Y_2)$, then $\Delta_1(g) = Y_1^{-1}Y_2^{-1/2}$ and $\Delta_2(g) = Y_1^{-1}Y_2^{-1}$. 

\begin{lem}\label{lem:1}
Let $w\in W$, $x\in U_w(\R)$, and $y, c, B \in \R_+^2$. Write $\y(\iota(B) c^*wx\iota(y)) = Y \in \R_+^2$ and $A = \iota(B)c^*$. Then we have
\ba
c_1 &\ll_{y,Y} B_1 B_2^{1/2}, & c_2 &\ll_{y, Y} B_1 B_2,
\ea
and
\ba
1\leq \Delta_1(wx) &\ll_{y,Y} \y(A)_1\y(A)_2^{1/2}, & 1\leq \Delta_2(wx) &\ll_{y,Y} \y(A)_1\y(A)_2. 
\ea
\end{lem}
\begin{proof}
Note that $\Delta_i (wx) \geq 1$ as one of its minors is always 1. For the first statement, we compute
\ba
\Delta_1(\iota(B) c^*wx\iota(y)) &= \frac{c_1}{B_1B_2^{1/2}} \Delta_1(wx\iota(y)), & \Delta_2(\iota(B) c^*wx\iota(y)) &= \frac{c_2}{B_1B_2} \Delta_2(wx\iota(y)).
\ea
Then we obtain
\ba
c_1\leq c_1\Delta_1(wx) \ll_y c_1\Delta_1(wx\iota(y)) &= \Delta_1(\iota(B)c^*wx\iota(y)) B_1B_2^{1/2} \ll_Y B_1B_2^{1/2},\\
c_2\leq c_2\Delta_2(wx) \ll_y c_2\Delta_2(wx\iota(y)) &= \Delta_2(\iota(B)c^*wx\iota(y)) B_1B_2 \ll_Y B_1B_2.
\ea
For the second statement, we observe that
\ba
c_1^{-1} &= \y(c^*)_1 \y(c^*)_2^{1/2}, & c_2^{-1} &= \y(c^*)_1 \y(c^*)_2.
\ea
Hence
\ba
\Delta_1(wx) \ll_{y,Y} B_1B_2^{1/2} c_1^{-1} &= (B\y(c^*))_1(B\y(c^*))_2^{1/2} = \y(A)_1\y(A)_2^{1/2},\\
\Delta_2(wx) \ll_{y,Y} B_1B_2 c_2^{-1} &= (B\y(c^*))_1 (B\y(c^*))_2 = \y(A)_1\y(A)_2,
\ea
finishing the proof.
\end{proof}

\begin{lem}\label{lem:2}
Let $N \in \N^2$ and $w\in W$. For $x\in U_w(\R)$, define $x' = \iota(N) x \iota(N)^{-1}$. Then
\ba
\frac{dx'}{dx} = (\tensor[^w]{N}{})^\eta N^\eta.
\ea
\end{lem}
\begin{proof}
By direct computation.
\end{proof}

\begin{lem}\label{lem:3}
Let $B \in \R_+^2$, and $w = s_\beta s_\alpha s_\beta$. Then
\ba
\vol\cbm{x\in U_{w}(\R)}{\Delta_j (wx) \leq B_j, j=1,2} \ll (B_1B_2)^{1+\varepsilon}
\ea
for any $\varepsilon>0$. 
\end{lem}
\begin{proof} 
We can assume without loss of generality that $B_j \geq 1$, otherwise we deduce from \Cref{lem:1} that the volume is 0. We have
\ba
x = \bp 1 && x_{13} & x_{23}\\ &1 & x_{23} & x_{24}\\ &&1\\ &&&1 \ep \in U_{w}(\R) \rsa wx = \bp &&&-1\\ &&1\\ &1 & x_{23} & x_{24}\\ -1 && -x_{13} & -x_{23}\ep. 
\ea
Then we obtain bounds 
\ba
\vb{x_{23}}&\leq B_1, \vb{x_{13}x_{24}-x_{23}^2} \leq B_2.
\ea
We also have $\vb{x_{13}}, \vb{x_{24}} \leq b :=  1 + \max\cb{B_1, B_2}$. If $I\sbe \R$ is any interval of length $\vb{I}\geq 1$, then
\ba
\vol\cbm{(x,y) \in [-b,b]^2}{xy\in I} \leq \int_{|y|\leq b} \min\cb{\frac{\vb{I}}{\vb{y}}, 2b} dy \leq 4\vb{I}(1+\log b).
\ea
Hence, if $\vb{x_{23}}\leq B_1$ is fixed, the volume of $(x_{13}, x_{24})$ is $O(B_2 \log b)$. This establishes the bound.
\end{proof}

\section{$\Sp(4)$ Kloosterman sums}
Properties of Kloosterman sums for $\Sp(4,\Z)$ were given in \cite{Man2020b}. They generalise in a natural way to the congruence subgroup $\Gamma_0(q)$. The Bruhat decomposition gives $\Sp(4,\Q) = \coprod\limits_{w\in W} G_w(\Q)$ with $G_w = UwTU_w$ as a disjoint union. Let $M, N, c\in \N^2$, $w\in W$. Then if
\begin{align}\label{eq:Kloosterman_well_defined}
\theta_M(wc^*x (c^*)^{-1} w^{-1}) = \theta_N(x)
\end{align}
for all $x\in U \cap w^{-1}U w$, the Kloosterman sum
\begin{align}\label{eq:Kloosterman_definition}
\Kl_{q,w}(c,M,N) := \sum\limits_{xwc^*x' \in U(\Z) \bs G_w(\Q) \cap \Gamma_0(q) / U_w(\Z)} \theta_M(x) \theta_N(x')
\end{align}
is well-defined \cite[Proposition 5.1]{Man2020b}. If \eqref{eq:Kloosterman_well_defined} does not hold, we set $\Kl_{q,w}(c,M,N) = 0$. The Kloosterman sum $\Kl_{q,w}(c,M,N)$ is nonzero only if $w = \id, s_\alpha s_\beta s_\alpha, s_\beta s_\alpha s_\beta, w_0$ \cite{Man2020b}. 

Now suppose the entries of $M = (M_1, M_2)$ and $N = (N_1, N_2)$ are coprime to $q$. Considering the Bruhat decomposition of $\Gamma_0(q)$, we deduce that the Kloosterman sum $\Kl_{q,w}(c,M,N)$ is nonempty only if 
\begin{align}\label{eq:Kl_nonempty}
q\mid c_1 &\text{ for } w = s_\alpha s_\beta s_\alpha, & q\mid c_1 \text{ and } q^2\mid c_2 &\text{ for } w = s_\beta s_\alpha s_\beta, w_0.
\end{align}
Meanwhile, the well-definedness condition \eqref{eq:Kloosterman_well_defined} says that the Kloosterman sums are well-defined precisely if
\ba
N_2 = M_2 \frac{c_1^2}{c_2^2} \text{ if } w &= s_\alpha s_\beta s_\alpha, & N_1 = M_1 \frac{c_2}{c_1^2} \text{ if } w&=s_\beta s_\alpha s_\beta.
\ea
Hence the Kloosterman sums are well-defined only if
\begin{align}\label{eq:Kl_q_defined}
v_q(c_1) = v_q(c_2) \text{ if } w &= s_\alpha s_\beta s_\alpha, & v_q(c_2) = 2v_q(c_1) \text{ if } w &= s_\beta s_\alpha s_\beta.
\end{align}

From the abstract definition \cite{Stevens1987, Man2020b}, the Kloosterman sums $\Kl_{q,w}(c,M,N)$ also enjoy certain multiplicativity in the moduli. We state one particular case. Let $q$ be prime. For $c = (c_1, c_2)\in\N^2$, let $c'_j = q^{-v_q(c_j)} c_j$, $j=1,2$, and $c' = (c'_1, c'_2)$. Then we have
\begin{align}\label{eq:Kl_mult}
\Kl_{q,w}(c, M, N) = \Kl_{q,w}\rb{(q^{v_q(c_1)}, q^{v_q(c_2)}}, M', N') \Kl_{1,w} (c', M'', N'')
\end{align}
for some $M', N', M'', N'' \in \N^2$. Moreover, if the entries of $M$, $N$ are coprime to $q$, then so are $M'$, $N'$. From \cite{DR1998}, we have a trivial bound
\begin{align}\label{eq:Kl_cprime_bound}
\Kl_{1,w} (c', M'', N'') \leq \vb{U(\Z) \bs G_w(\Q) \cap \Sp(4,\Z) / U_w(\Z)} \leq c'_1 c'_2.
\end{align}

\subsection{Evaluations of Kloosterman sums}\label{section:Kl_eval}

For the proof of the theorems in \Cref{section:proofs}, we compute the following Kloosterman sums:
\ba
&\Kl_{q,s_\alpha s_\beta s_\alpha}\rb{(q,q),M,N}, & &\Kl_{q,s_\beta s_\alpha s_\beta}\rb{(q,q^2),M,N},\\
&\Kl_{q,w_0}\rb{(q,q^2),M,N}, & &\Kl_{q,w_0}\rb{(q,q^3),M,N}.
\ea
\ber
\item 
Consider the Bruhat decomposition for summands in $\Kl_{q, s_\alpha s_\beta s_\alpha} \rb{(q,q), M, N}$:
{\small
\ba
\gamma = &\bp 1 & \beta_1 & \beta_2 & \beta_3\\ & 1 & \beta_4 & \beta_5\\ &&1\\&&-\beta_1 & 1\ep \bp &&-q^{-1}\\&1\\q\\&&&1\ep \bp 1 & v_2/q & v_3/q & v_4/q\\ & 1 & v_4/q \\ &&1\\&&-v_2/q & 1\ep\\
= & \bp \beta_2 q & \beta_2 v_2 + \beta_1 & \beta_2 v_3 - \beta_3 v_2/q + \beta_1 v_4 /q - 1/q & \beta_2 v_4 + \beta_3\\
\beta_4 q & \beta_4 v_2+ 1 & \beta_4 v_3 - \beta_5 v_2/q + v_4/q & \beta_4 v_4 + \beta_5\\
q & v_2 & v_3 & v_4\\
-\beta_1 q & -\beta_1 v_2 & -\beta_1 v_3 - v_2/q & -\beta_1 v_4 + 1\ep \in \Gamma_0(q),
\ea}with $v_2, v_3, v_4 \pmod{q}$ chosen such that $\rb{v_3, v_4, (q, v_2)} = 1$, and $\rb{(q,v_2)^2, qv_3+v_2v_4} = q$. As $\gamma \in \Gamma_0(q)$, by considering the lower left block, we deduce that $v_2 = 0$, and solve $\beta_1 \equiv 0 \pmod{1}$. The conditions on $v_3, v_4$ then simplify as $(q,v_3) = 1$. Considering the second row, we solve
\ba
\beta_4 &\equiv -\frac{\ol{v_3} v_4}{q} \pmod{1}, & \beta_5 &\equiv \frac{\ol{v_3}v_4^2}{q} \pmod{1}.
\ea
So the Kloosterman sum is given by
\ba
\Kl_{q, s_\alpha s_\beta s_\alpha} \rb{(q,q), M, N} = \sum\limits_{\substack{v_3\ppmod{q}\\ (v_3, q) = 1}} \sum\limits_{v_4 \ppmod{q}} e\rb{\frac{M_2 \ol{v_3} v_4^2}{q}} = 0.
\ea

\item
Consider the Bruhat decomposition for summands in $\Kl_{q, s_\beta s_\alpha s_\beta} \rb{(q,q^2), M, N}$:
{\small
\ba
\gamma = &\bp 1 & \beta_1 & \beta_2 & \beta_3\\ & 1 & \beta_4 & \beta_5\\ &&1\\&&-\beta_1 & 1\ep \bp &&& -q^{-1}\\ &&q^{-1}\\ &q\\ -q\ep \bp 1 && -v_{23}/q^2 & v_{13}/q^2\\ &1& v_{13}/q^2 & v_{14}/q^2\\ &&1\\ &&&1\ep\\
= &\bp -\beta_3 q & \beta_2 q & \beta_2 v_{13}/q + \beta_1/q + \beta_3 v_{23}/q & \beta_2 v_{14}/q - \beta_3 v_{13}/q - 1/q\\
-\beta_5 q & \beta_4 q & \beta_4 v_{13}/q + \beta_5 v_{23}/q + 1/q & \beta_4 v_{14}/q - \beta_5 v_{13}/q\\
0 & q & v_{13}/q & v_{14}/q\\
-q & -\beta_1 q & -\beta_1 v_{13}/q + v_{23}/q & -\beta_1 v_{14}/q - v_{13}/q\ep \in \Gamma_0(q),
\ea}with $v_{13}, v_{14}, v_{23} \pmod{q^2}$ chosen such that $(q^2, v_{13}, v_{14}) = q$, and $(q, v_{23}, v_{34}) = 1$, where $v_{34} = -\frac{v_{13}^2 + v_{14}v_{23}}{q^2}$. As $\gamma\in\Gamma_0(q)$, by considering the lower left block, we solve $\beta_1\equiv 0\pmod{1}$. Then, $-\beta_1 v_{13} /q + v_{23}/q$ being an integer implies $q \mid v_{23}$, so $(q,v_{34})=1$. Write $v_{13} = q v'_{13}$, $v_{14} = q v'_{14}$, and $\beta_4 = \beta'_4/q$, $\beta_5 = \beta'_5/q$ for some $\beta'_4, \beta'_5\in\Z$. By considering the second row, we deduce that
\ba
\beta'_4 v'_{13} + \beta'_5 v_{23}/q + 1,  \beta'_4 v'_{14} - \beta'_5 v'_{13} &\in q\Z,
\ea
from which we deduce $\beta'_5 \equiv v'_{14} \ol{v_{34}} \pmod{q}$, and $\beta_5 \equiv \frac{v'_{14}\ol{v_{34}}}{q} \pmod{1}$. Writing $v_{23} = q v'_{23}$, the Kloosterman sum is given by
{\small
\ba
\Kl_{q, s_\beta s_\alpha s_\beta} \rb{(q,q^2), M, N} = \sum\limits_{\substack{v'_{13}, v'_{14}, v'_{23} \ppmod{q}\\ (q, v'_{13}, v'_{14}) = 1\\ (q, v_{34}) = 1}} e\rb{\frac{M_2 v'_{14}\ol{v_{34}}+N_2v'_{14}}{q}},
\ea}
where $v_{34} = -({v'_{13}}^2+v'_{14} v'_{23})$. We evaluate
{\small
\ba
&\sum\limits_{\substack{v'_{13}, v'_{23} \ppmod{q}\\ (q, v'_{13}) = 1}} 1 + \sum\limits_{\substack{v'_{13}, v'_{14}, v'_{23}\ppmod{q}\\ (q, v'_{14}) = 1,\: (q, v_{34}) = 1}} e\rb{\frac{M_2v'_{14}\ol{v_{34}}+N_2v'_{14}}{q}} = & q(q-1) - \sum\limits_{\substack{v'_{13}, v'_{14} \ppmod{q}\\ (q, v'_{14}) = 1}} e\rb{\frac{N_2v'_{14}}{q}} = q^2.
\ea}

\item
Consider the Bruhat decomposition for summands in $\Kl_{q, w_0} \rb{(q,q^2), M, N}$:
{\small\ba
\gamma = &\bp 1 & \beta_1 & \beta_2 & \beta_3\\ & 1 & \beta_4 & \beta_5\\ &&1\\&&-\beta_1 & 1\ep \bp && -q^{-1}\\ &&&-q^{-1}\\ q\\ &q\ep \bp 1 & v_2/q & v_3/q & v_4/q \\ & 1 & v_{13}/q^2 & v_{14}/q^2 \\ &&1\\ &&-v_2/q & 1\ep\\
= &
\bp \beta_2 q & \beta_2 v_2 + \beta_3 q & \beta_2 v_3 + \beta_3 v_{13}/q + \beta_1 v_2/q^2 - 1/q & \beta_2 v_4 - \beta_1/q + \beta_3 v_{14}/q\\
\beta_4 q & \beta_4 v_2 + \beta_5 q & \beta_4 v_3 + \beta_5 v_{13}/q + v_2/q^2 & \beta_4 v_4 + \beta_5 v_{14}/q - 1/q\\
q & v_2 & v_3 & v_4\\
-\beta_1 q & -\beta_1 v_2 + q & -\beta_1 v_3 + v_{13}/q & -\beta_1 v_4 + v_{14}/q\ep \in \Gamma_0(q),
\ea}with $v_2, v_3 ,v_4 \pmod{q}$, $v_{13}, v_{14} \pmod{q^2}$ chosen such that $v_{13} q + v_2 v_{14} - v_4 q^2 = 0$, $(q, v_2, v_3, v_4) = 1$, and $(q^2, v_{13}, v_{14}, v_{23}, v_{34}) = 1$, where $v_{23} = \frac{v_2v_{13}-v_3 q^2}{q}$ and $v_{34} = \frac{v_3v_{14}-v_4v_{13}}{q}$. As $\gamma \in \Gamma_0(q)$, by considering the lower left block, we deduce that $v_2 = 0$, and solve $\beta_1 \equiv 0 \pmod{1}$. The last row being integers implies that $q \mid v_{13}, v_{14}$. Write $v_{13} = q v'_{13}$, $v_{14} = qv'_{14}$. The relation $v_{13}q + v_2 v_{14} - v_4 q^2 = 0$ says $v'_{13} = v_4$. We check that $q \mid v_{23}$ as well, so $(q, v_{34}) = 1$. Write $\beta_4 = \beta'_4/q$, $\beta_5 = \beta'_5/q$ for some $\beta'_4, \beta'_5\in\Z$. By considering the second row, we deduce that
\ba
\beta'_4 v_3 + \beta'_5 v'_{13}, \; \beta'_4 v_4 + \beta'_5 v'_{14} - 1 \in q\Z,
\ea
from which we deduce $\beta'_5 \equiv v_3 \ol{v_{34}} \pmod{q}$, and $\beta_5 = \frac{v_3 \ol{v_{34}}}{q} \pmod{1}$. The Kloosterman sum is given by
{\small
\ba
\Kl_{q, w_0} \rb{(q,q^2), M, N} = \sum\limits_{\substack{v_3, v_4, v'_{14}\ppmod{q}\\ (q, v_3, v_4) = 1\\ (q, v_{34}) = 1}} e\rb{\frac{M_2v_3\ol{v_{34}} + N_2v'_{14}}{q}},
\ea}
where $v_{34} = v_3 v'_{14} - v_4^2$. 

Fix $v_4, v'_{14} \neq 0$. As $v_3\neq 0$ varies, $\ol{v_3} v_{34} \equiv v'_{14} - v_4^2\ol{v_3}$ runs through nonzero residues except $v'_{14}$ modulo $q$; hence, as $v_3$ varies, $v_3 \ol{v_{34}}$ runs through all residues except $\ol{v'_{14}}$ modulo $q$. Hence
{\small
\ba
\sum\limits_{\substack{v_3, v_4, v'_{14}\ppmod{q}\\ (q, v_4) = 1, \; (q, v'_{14}) = 1\\ (q, v_{34}) = 1}} e\rb{\frac{M_2v_3\ol{v_{34}} + N_2v'_{14}}{q}} = -\sum\limits_{\substack{v_4, v'_{14}\ppmod{q}\\ (q, v_4) = 1\\ (q, v'_{14})=1}} e\rb{\frac{M_2\ol{v'_{14}}+N_2v'_{14}}{q}} = -(q-1) S(M_2,N_2;q).
\ea}If $v_4\neq 0$ and $v'_{14} = 0$, then $v_{34} = -v_4^2$. The corresponding part of the sum becomes
\ba
\sum\limits_{\substack{v_3, v_4\ppmod{q}\\ (q, v_4) = 1}} e\rb{\frac{-M_2v_3\ol{v_4}^2}{q}} = 0.
\ea
Meanwhile, for $v_4 = 0$, we have $v'_{14} \neq 0$, and $v_{34} = v_3 v'_{14}$, so $v_3 \ol{v_{34}} = \ol{v'_{14}}$. Hence this part of the sum is
{\small
\ba
\sum\limits_{\substack{v_3, v'_{14} \ppmod{q}\\ (q, v_3) = 1 \\ (q, v'_{14}) = 1}} e\rb{\frac{M_2\ol{v'_{14}} + N_2v'_{14}}{q}} = (q-1) S(M_2,N_2;q).
\ea}Combining the parts above, we conclude that $\Kl_{q, w_0} \rb{(q,q^2), M, N} = 0$.

\item
Consider the Bruhat decomposition for summands in $\Kl_{q,w_0}\rb{(q,q^3),M,N}$:
{\small
\ba
\gamma = &\bp 1 & \beta_1 & \beta_2 & \beta_3\\ & 1 & \beta_4 & \beta_5\\ &&1\\&&-\beta_1 & 1\ep \bp && -q^{-1}\\ &&&-q^{-2}\\ q\\ &q^2\ep \bp 1 & v_2/q & v_3/q & v_4/q \\ & 1 & v_{13}/q^3 & v_{14}/q^3 \\ &&1\\ &&-v_2/q & 1\ep\\
= &
\bp \beta_2 q & \beta_2 v_2 + \beta_3 q^2 & \beta_2 v_3 + \beta_3 v_{13}/q + \beta_1v_2/q^3 - 1/q & \beta_2v_4-\beta_1/q^2+\beta_3v_{14}/q\\
\beta_4 q & \beta_4 v_2 + \beta_5 q^2 & \beta_4 v_3 + \beta_5 v_{13}/q + v_2/q^3 & \beta_4 v_4 + \beta_5 v_{14}/q - 1/q^2\\
q & v_2 & v_3 & v_4\\
-\beta_1 q & -\beta_1 v_2 + q^2 & -\beta_1 v_3 + v_{13}/q & -\beta_1 v_4 + v_{14}/q\ep \in \Gamma_0(q),
\ea}with $v_2, v_3, v_4 \pmod{q}$, $v_{13}, v_{14} \pmod{q^2}$ chosen such that $v_{13}q+v_2v_{14}-v_4q^3=0$, $(q,v_2,v_3,v_4) = 1$, and $(q^2, v_{13}, v_{14}, v_{23}, v_{34}) = 1$, where $v_{23} = \frac{v_2v_{13}-v_3q^3}{q}$ and $v_{34} = \frac{v_3v_{14}-v_4v_{13}}{q}$. As $\gamma \in \Gamma_0(q)$, by considering the lower left block, we deduce that $v_2 = 0$, and solve $\beta_1\equiv 0 \pmod{1}$. The last row being integers implies that $q \mid v_{13}, v_{14}$. Write $v_{13} = qv'_{13}$, $v_{14} = qv'_{14}$. The relation $v_{13}q+v_2v_{14}-v_4q^3=0$ says $v'_{13} = v_4 q$. We check that $q \mid v_{23}$ as well, so $(q,v_{34}) = 1$. Write $\beta_4 = \beta'_4/q$, $\beta_5 = \beta'_5/q^2$ for some $\beta'_4, \beta'_5\in\Z$. By considering the second row, we deduce that
\ba
\beta'_4v_3q+\beta'_5v'_{13}, \; \beta'_4v_4q + \beta'_5v'_{14}-1 \in q^2\Z,
\ea
from which we deduce $\beta'_5 \equiv v_3 \ol{v_{34}} \pmod{q^2}$, and $\beta_5 = \frac{v_3\ol{v_{34}}}{q^2} \pmod{1}$. The Kloosterman sum is given by
\ba
\Kl_{q,w_0}\rb{(q,q^3),M,N} = \sum\limits_{\substack{v_3,v_4\ppmod{q}, \; v'_{14}\ppmod{q^2}\\(q,v_3,v_4)=1, \;(q,v_{34}) = 1}} e\rb{\frac{M_2v_3\ol{v_{34}} + N_2v'_{14}}{q^2}},
\ea
where $v_{34} = v_3v'_{14} - v_4^2 q$. 

Fix $v_4\neq 0$. Then from $(q,v_{34}) = 1$ we have $(q,v'_{14}) = 1$, and $v_3 \neq 0$. For a fixed $v'_{14}$, we see that as $v_3$ varies, $\ol{v_3} v_{34} \equiv v'_{14} - v_4^2 \ol{v_3} q$ runs through nonzero residues modulo $q^2$ that are congruent to $v'_{14}\pmod{q}$, except $v'_{14}$; hence, as $v_3$ varies, $v_3\ol{v_{34}}$ runs through all residues modulo $q^2$ that are congruent to $v'_{14}\pmod{q}$, except $\ol{v'_{14}}$. Hence
{\small
\ba
\sum\limits_{\substack{v_3,v_4\ppmod{q}\\ v'_{14} \ppmod{q^2}\\ (q,v_3) = 1, \; (q,v_4)=1\\ (q,v'_{14}) = 1, \; (q,v_{34}) = 1}} e\rb{\frac{M_2v_3\ol{v_{34}} + N_2v'_{14}}{q^2}} = - \sum\limits_{\substack{v_4\ppmod{q}\\ v'_{14} \ppmod{q^2}\\ (q,v_4)=1,\; (q,v'_{14}) = 1}} e\rb{\frac{M_2 \ol{v'_{14}} + N_2 v'_{14}}{q^2}} = -(q-1) S(M_2,N_2;q^2).
\ea}Meanwhile, for $v_4=0$, we have $(q,v'_{14})=1$, and $v_{34} = v_3v'_{14}$, so $v_3\ol{v_{34}} = \ol{v'_{14}}$. Hence this part of the sum is
{\small
\ba
\sum\limits_{\substack{v_3\ppmod{q}\\ v'_{14}\ppmod{q^2}\\ (q,v_3) = 1, \; (q, v'_{14}) = 1}} e\rb{\frac{M_2\ol{v'_{14}}+N_2v'_{14}}{q^2}} = (q-1) S(M_2,N_2;q).
\ea}Combining the parts above, we conclude that $\Kl_{q,w_0}\rb{(q,q^3), M,N} = 0$.
\ee

\section{Automorphic forms and Whittaker functions}

We denote by $\cb{\varpi}$ an orthonormal basis of right $K$-invariant automorphic forms for $\Gamma_0(q)$, cuspidal or Eisenstein series. The space $L^2(\Gamma_0(q) \bs \Sp(4,\R)/ K)$ is equipped with the standard inner product 
\ba
\pb{f,g} = \int_{\Gamma_0(q)\bs \Sp(4,\R)/K} f(xy) \ol{g(xy)} dx d^*y.
\ea
An integral over the complete spectrum of $L^2(\Gamma_0(q)\bs \Sp(4,\R)/K)$ is denoted by $\int_{(q)} d\varpi$. All the automorphic forms $\varpi$ belong to representations $\pi$ of level $q' \mid q$, and we assume that $\cb{\varpi}$ contains all cuspidal newvectors of level $q' \mid q$. For simplicity in notations, we denote the local archimedean spectral parameter $\mu_\pi(\infty)$ by $\mu = (\mu_1, \mu_2)$.

Let $\varpi$ be an automorphic form for $\Gamma_0(q)$, with spectral parameter $\mu$. We suppose $\varpi$ is generic throughout the section. For $M = (M_1, M_2) \in \Z^2$, the $M$-th Fourier coefficient of $\varpi$ is given by
\ba
\varpi_M (g) = \int_{U(\Z)\bs U(\R)} \varpi(xg) \ol{\theta_M(x)} dx.
\ea
The Fourier coefficients $\varpi_M(g)$ are actually Whittaker functions. For $g = xyk \in\Sp(4,\R)$, we have
\begin{align}\label{eq:Fourier_in_Whittaker}
\varpi_M(g) = \frac{A_\varpi(M)}{M^\eta} \theta_M(x) \cdot W_\mu(\iota(M)y),
\end{align}
where $A_\varpi(M) \in \C$ is a constant, called the $M$-th Fourier coefficient of $\varpi$, $W_\mu: \R_+^2 \to \C$ is the standard Whittaker function on $\Sp(4,\R)$, with detailed descriptions found in \cite{Ishii2005}. As in the $\GL(n)$ case, the size of $\sigma_\pi(\infty)$ captures the growth of $W_\mu$ near the origin. Precisely, for a function $E$ on $\R_+^2$ and $X \in \R_+^2$, we define
\begin{align}\label{eq:E_def}
E^{(X)}(y_1, y_2) = E(X_1y_1, X_2y_2).
\end{align}

\begin{lem}\label{lem:5}
Assume that $\mu = (\mu_1, \mu_2)$ varies in some compact set $\Omega$, and let $Z\geq 1$. There exists $r\in \N$ and a compact set $S\sbe \R_+^2$ depending only on $\Omega$ (independent of $Z$), and a finite collection of functions $E_1, \cdots, E_r: \R_+^2\to \R$ depending on $\Omega$ and $Z$ that are uniformly bounded and supported in a compact subset of $S$ such that
\ba
\sum\limits_{j=1}^r \vb{\pb{E_j^{(1,Z)}, W_\mu}}^2 \gg_\Omega Z^{2\eta_2+2\sigma_\pi(\infty)} = Z^{3+2\sigma_\pi(\infty)}. 
\ea
\end{lem}
\begin{proof}
The case $Z\ll 1$ is proved in \cite{BBM2017, Blomer2019b}. For each $\mu \in \Omega$, choose an open set $S_\mu\sbe \R_+^2$ such that $\Re W_\mu(y) \neq 0$ for all $y\in S_\mu$ or $\Im W_\mu(y) \neq 0$ for all $y\in S_\mu$. Now choose open neighbourhoods $U_\mu$ about $\mu$ such that $\Re W_{\mu^*}(y) \neq 0$ for all $y\in S_\mu$ and $\mu^* \in U_\mu$, or $\Im W_{\mu^*}(y) \neq 0$ for all $y\in S_\mu$ and $\mu^* \in U_\mu$. By compactness, $\Omega$ is covered by a finite collection of neighbourhoods $U_{\mu_1}, \cdots, U_{\mu_r}$, and we may pick corresponding $E_j$ to be real-valued functions with supports on $S_{\mu_j}$ and non-vanishing on the interior $S_{\mu_j}^\circ$. 

Now suppose $Z\gg 1$ is sufficiently large. Consider the following renormalisation of the Whittaker function:
\begin{align}\label{eq:W_Wstar}
W_\mu^*(y) := y^{-\eta} W_\mu(y). 
\end{align}
The Mellin transform $M_\mu^*(s) = \int_{\R_+^2} W_\mu^*(y) y^s \frac{dy_1}{y_1} \frac{dy_2}{y_2}$ is given by \cite{Ishii2005} (where $\nu_1, \nu_2$ in \cite{Ishii2005} are $\mu_1+\mu_2$ and $\mu_1-\mu_2$ in our notation)
{\small 
\ba
M_\mu^*(s) = &2^{-4} \Gamma\rb{\frac{s_1+\mu_1+\mu_2}{2}}\Gamma\rb{\frac{s_1+\mu_1-\mu_2}{2}}\Gamma\rb{\frac{s_1-\mu_1+\mu_2}{2}}\Gamma\rb{\frac{s_1-\mu_1-\mu_2}{2}}\\
&\Gamma\rb{\frac{s_2+\mu_1}{2}}\Gamma\rb{\frac{s_2-\mu_1}{2}}\Gamma\rb{\frac{s_2+\mu_2}{2}}\Gamma\rb{\frac{s_2-\mu_2}{2}}\\
&\cb{\Gamma\rb{\frac{s_1+s_2+\mu_1}{2}}\Gamma\rb{\frac{s_1+s_2-\mu_1}{2}}}^{-1} \tensor[_3]{F}{_2}\left( \begin{array}{c} \frac{s_1}{2}, \frac{s_2+\mu_2}{2}, \frac{s_2-\mu_2}{2}\\ \frac{s_1+s_2+\mu_1}{2}, \frac{s_1+s_2-\mu_1}{2} \end{array} \middle| 1 \right).
\ea}
By Weyl group symmetry, we may assume without loss of generality that $\sigma_\pi(\infty) = \Re \mu_1$. For $\Re(s_1)$ sufficiently large, $M_\mu^*(s)$ is holomorphic for $\Re(s_2) > \sigma_\pi(\infty)$, as poles can only occur at $s_2 = \pm \mu_1-k$, $\pm \mu_2-k$ for $k\in\N_0$. Hence, for $\Re(s_1)$ sufficiently large, the function
\ba
M_\mu^\dagger (s) := M_\mu^*(s) (s_2+\mu_1)(s_2-\mu_1)(s_2+\mu_2)(s_2-\mu_2)
\ea
is holomorphic for $\Re(s_2) > \sigma_\pi(\infty) - 1$.

For $\beta\in\C$, let $\mc D_\beta = -y_2\partial_{y_2} + \beta$. This is a commutative family of differential operators, which correspond to multiplication by $s_2+\beta$ under Mellin transform. Now let
\ba
\hat M_\mu(s) := \frac{M_\mu^\dagger (s)}{s_2-\mu_1} = M_\mu^*(s) (s_2+\mu_1)(s_2+\mu_2)(s_2-\mu_2).
\ea
Taking inverse Mellin transforms, we get
\ba
\hat W_\mu(y) = \mc D_{\mu_1} \mc D_{\mu_2} \mc D_{-\mu_2} W_\mu^*(y).
\ea
On the other hand, we compute the inverse Mellin transform directly, and by shifting the contour to $\Re(s_2) = \sigma_\pi(\infty) - \frac{1}{2}$, we obtain the estimate
\ba
\hat W_\mu^*(y) = y_2^{-\mu_1} W_\mu^{**} (y_1) + \mc O_{y_1,\mu}(y_2^{-\sigma_\pi(\infty)+\frac{1}{2}})
\ea
for $y_2\to 0$, where
\ba
W_\mu^{**}(y_1) = \Gamma\rb{\mu_1+1} \Gamma\rb{\frac{\mu_1+\mu_2}{2}+1} \Gamma\rb{\frac{\mu_1-\mu_2}{2}+1} W_{\mu_2}^* (y_1) y_1^{-\mu_1},
\ea
where $W_{\mu_2}^* (y_1) = y_1^{-1/2} W_{\mu_2} (y_1)$ is a normalised $\GL(2)$-Whittaker function.

The rest of the proof follows the argument in \cite{Blomer2019b}. We see that if $\mc D_\beta w(y) \sim cy^{-a}$ as $y\to 0$ for some constants $a\geq 0$, $\beta, c \in \C$ and $Z\geq 1$ is sufficiently large, then there exist constants $0<\gamma_1<\gamma_2<1$ (depending on all parameters, but uniformly bounded away from 0 when $\beta, c, a$ vary in a fixed compact set and $Z$ is sufficiently large) such that $\vb{w(y)} \geq \frac{1}{2} cy^{-a}$ for $y\in [\gamma_1/Z, \gamma_2/Z]$. Iterating this argument, and adjust the constants $\gamma_1, \gamma_2$ if necessary, we deduce that
\ba
\vb{W_\mu^* (y)} \gg y_2^{-\sigma_\pi(\infty)} \vb{W_\mu^{**} (y_1)}
\ea
for $y\in[\gamma_1/Z, \gamma_2/Z]$, when $y_2$ and $\mu$ vary in some fixed compact domain. Now choose functions $E_j^{*}: \R^+ \to \C$, depending on $\Omega$ but not $Z$, such that $\sum_j \vb{\pb{E_j^{*}, W_\mu^{**}}}^2 \gg 1$ for $\mu\in\Omega$. Now define $E_j(y_1, y_2) = \delta_{\gamma_1\leq y_2\leq \gamma_2} E_j^{*} (y_1)$. This choice depends on $Z$, but the support of $E$ varies inside some interval depending only on $\Omega$. We then obtain
\ba
\sum\limits_j \vb{\pb{E_j^{(1,Z)}, W_\mu^*}}^2 \gg Z^{2\sigma_\pi(\infty)}.
\ea
Using the relation \eqref{eq:W_Wstar}, we obtain the lemma.
\end{proof}

\section{Hecke eigenvalues and Fourier coefficients}\label{section:Hecke_Fourier}

Let $\mc M$ be a set of matrices in $\GSp(4,\Q)^+$ that is left- and right-invariant under $\Gamma = \Sp(4,\Z)$ and is a finite union $\bigcup\limits_j \Gamma \mc M_j$ of left cosets. Then $\mc M$ defines a Hecke operator $T_{\mc M}$ on the space of cuspidal automorphic forms by
\ba
T_{\mc M} \varpi(g) = \sum\limits_j \varpi(\mc M_j g).
\ea
For a matrix $g\in \GSp(4,\Q)^+$, we denote by $T_g$ the Hecke operator $T_{\Gamma g \Gamma}$. For $m\in \N$, let 
\ba
S(m) :&= \cbm{M \in \GSp(4,\Z)^+}{M^\top J M = mJ}, & J = \bp &I_2\\-I_2\ep.
\ea
The $m$-th standard Hecke operator is then given by $T(m) = T_{S(m)}$. The set of matrices
\begin{align}\label{eq:Sm_coset_rep}
{\small \mc H(m) = \cbm{\bp A & m^{-1} BD\\ & D\ep\in S(m)}{A = \bp a_1 & a\\ & a_2\ep, B = \bp b_1 & b_2\\ b_2 & b_3\ep, \begin{array}{l} a_1, a_2 > 0, \;0\leq a < a_2,\\ 0\leq b_i < m, AD^\top = mI_2,\\ BD\equiv 0\pmod{m}\end{array}}.}
\end{align}
gives a complete system of left coset representatives for $\Gamma \bs S(m)$ \cite{Spence1972}. For $r\in\N_0$, $0\leq a \leq b \leq r/2$ and any prime $p$, define
\ba
T_{a,b}^{(r)} (p) := T_{\diag(p^a, p^b, p^{r-a}, p^{r-b})}.
\ea
When the context is clear, we suppress $p$ from the notation, and write $T_{a,b}^r$ instead. Then $T(p^r)$ admits a decomposition
\ba
T(p^r) = \sum\limits_{0\leq a \leq b \leq r/2} T_{a,b}^{(r)}(p). 
\ea
It is well-known that the Hecke algebra $\ms H$ of $\Sp(4,\R)$ is generated by $T(p) = T_{0,0}^{(1)}(p)$ and $T_{0,1}^{(2)}(p)$ for primes $p$, along with the identity. 

We also define involutions $T_{\varepsilon_1}$, $T_{\varepsilon_2}$ on the space of cuspidal automorphic forms by
\ba
T_{\varepsilon_1}\varpi(g) &= \varpi(\varepsilon_1 g), & T_{\varepsilon_2}\varpi\rb{\bp Y & X\\ & (Y^{-1})^\top\ep} &= \varpi\rb{\bp Y & -X\\ & (Y^{-1})^\top\ep}, 
\ea
where $\varepsilon_1 = \diag(-1,1,-1,1)$. It is clear that
\begin{align}\label{eq:Teps_action}
(T_{\varepsilon_1} \varpi)_{(M_1, M_2)}(g) &= \varpi_{(-M_1,M_2)}(g), & (T_{\varepsilon_2} \varpi)_{(M_1, M_2)}(g) &= \varpi_{(M_1,-M_2)}(g).
\end{align}
It is also straightforward to check that $T_{\varepsilon_1}$, $T_{\varepsilon_2}$ commute with the Hecke operators and the invariant differential operators. So we may assume a cuspidal automorphic form $\varpi$ is also an eigenfunction of $T_{\varepsilon_1}$, $T_{\varepsilon_2}$. 

Let $\pi$ be the irreducible automorphic representation corresponding to $\varpi$. We write $\lambda(m, \pi)$ and $\lambda_{a,b}^{(r)}(p, \pi)$ to denote the eigenvalue of $\varpi$ with respect to $T(m)$ and $T_{a,b}^{(r)}$ respectively, and write $\lambda'(m, \pi) := m^{-3/2} \lambda(m,\pi)$. Again, we omit $\pi$ from the notation when the context is clear. 

It is known that if $\varpi$ is generic and $L^2$-normalised, then by \cite[Theorem 1.1]{CI2019} and \cite[Theorem 3]{Li2010}, we have
\begin{align}\label{eq:A11_estimate}
\vb{A_\varpi(1,1)}^2 \asymp_\mu \frac{1}{[\Sp(4,\Z): \Gamma_0(q)] L(1,\pi, \Ad)} \gg_\mu q^{-3-\varepsilon}.  
\end{align}
In particular, $A_\varpi(1,1) \neq 0$. 

\begin{ntn}
Let $\varpi$ be an $L^2$-normalised generic cuspidal newform. For the rest of the section, it is however instructive to have an alternative normalisation, such that the $(1,1)$-st Fourier coefficient is 1. To avoid confusion, we always denote by $\varpi$ an $L^2$-normalised form, and by $\varpi_1$ a scalar multiple of $\varpi$ such that $A_{\varpi_1}(1,1) = 1$. From \eqref{eq:A11_estimate}, we see that $\varpi_1 = k\varpi$ for some $\vb{k}\ll q^{3+\varepsilon}$. 
\end{ntn}

Now fix a prime $p \nmid q$. Let $M = (M_1, M_2)$, and $0 \leq c, d \leq r$ such that $p^{d-c} \mid M_1$ and $p^{r-2d} \mid M_2$. Write 
\ba
\Gamma \diag(p^a, p^b, p^{r-a}, p^{r-b}) \Gamma = \bigcup\limits_i \Gamma h_i
\ea
as a finite union of left cosets. We can assume that $h_i \in U(\Q)T(\Q_+)$. Consider the decomposition $h_i = \hat y_i \hat x_i$, with $\hat y_i \in T(\Q^+)$, $\hat x_i \in U(\Q^+)$. We define exponential sums
\ba
\mf S_{a,b,M}^{(r)} (c,d) := \sum\limits_{\substack{\Gamma h_i \sbe \Gamma \diag(p^a, p^b, p^{r-a}, p^{r-b}) \Gamma\\ \hat y_i = \diag (p^c, p^d, p^{r-c}, p^{r-d})}} \theta_M(\hat x_i),
\ea
and
\ba
\mf S^{(r)} (c,d) := \sum\limits_{0\leq a, b\leq r/2} \mf S_{a,b,(1,1)}^{(r)} (c,d) = \sum\limits_{\substack{\Gamma h_i \sbe S(p^r)\\ \hat y_i = \diag(p^c, p^d, p^{r-c}, p^{r-d})}} \theta(\hat x_i). 
\ea

\begin{prp}\label{prp:Hecke_equation}
We have
\ba
\lambda_{a,b}^{(r)}(p) A_\varpi(M) = \sum\limits_{\substack{0\leq c, d \leq r\\ p^{c-d} \mid M_1, p^{2d-r} \mid M_2}} \mf S_{a,b,M}^{(r)} (c,d) p^{2c+d-\frac{3r}{2}} A_\varpi(M_1 p^{d-c}, M_2 p^{r-2d}).
\ea
\end{prp}
\begin{proof}
We compute the Fourier coefficient of $T_{a,b}^{(r)} \varpi$ in two ways. On one hand, we have
\begin{align}\label{eq:Hecke_eq1}
\int_{U(\Z)\bs U(\R)} T_h \varpi(xy) \ol{\theta_M(x)} dx = \lambda_{a,b}^{(r)}(p) \frac{A_\varpi(M)}{M^\eta} W_\mu(\iota(M)y).
\end{align}
On the other hand, we expand the Hecke operator
\ba
\int_{U(\Z)\bs U(\R)} T_h \varpi(xy) \ol{\theta_M(x)} dx = &\sum\limits_{\Gamma h_i} \int_{U(\Z)\bs U(\R)} \varpi(h_i xy) \ol{\theta_M(x)} dx\\
= & p^{-4r} \sum\limits_{\Gamma h_i} \int_{U(p^r \Z)\bs U(\R)} \varpi(h_i xy) \ol{\theta_M(x)} dx.
\ea
Write $h_i x = x' \hat y_i$, with $x' \in U(\R)$, and $\hat y_i = \diag (c_1, \cdots, c_4)$. A simple calculation shows that
\ba
x'_{k,l} = c_l \sum\limits_j (h_i)_{kj} x_{jl}. 
\ea
In particular, we have
\ba
x_{12} &= \frac{c_2}{c_1} x'_{12} - \frac{(h_i)_{12}}{c_1} = \frac{c_2}{c_1}x'_{12} - (\hat x_i)_{12}, & x_{24} &= \frac{c_4}{c_2} x'_{24} - \frac{(h_i)_{24}}{c_2} = \frac{c_4}{c_2} x'_{24} - (\hat x_i)_{24}.
\ea
Making this substitution, the expression becomes
\ba
p^{-4r} \sum\limits_{\Gamma h_i} \prod\limits_{k,l} \int_{\sum\limits_j (h_i)_{kj}x_{jl}}^{\frac{c_k}{c_l}p^r+\sum\limits_j (h_i)_{kj}x_{jl}} \varpi(x' \hat y_i y) e\rb{M_1 (\hat x_i)_{12} + M_2 (\hat x_i)_{24}} e\rb{-\frac{c_2}{c_1} M_1 x'_{12} - \frac{c_4}{c_2} M_2 x'_{24}} \frac{c_l}{c_k} dx'_{k,l},
\ea
where $(k,l)$ runs through the indices $(1,2), (1,3), (2,3), (2,4)$. 
 By periodicity, we shift the integral and get
\ba
p^{-4r} \sum\limits_{\Gamma h_i} \prod\limits_{k,l} \int_0^{\frac{c_k}{c_l}p^r} \varpi(x' \hat y_i y) e\rb{M_1 (\hat x_i)_{12} + M_2 (\hat x_i)_{24}} e\rb{-\frac{c_2}{c_1} M_1 x'_{12} - \frac{c_4}{c_2} M_2 x'_{24}} \frac{c_l}{c_k} dx'_{k,l},
\ea
Since $\varpi(x' \hat y_i y)$ is 1-periodic with respect to $x'_{kl}$, this integral vanishes unless $c_1\mid c_2 M_1$ and $c_2 \mid c_4 M_2$. We sum over the terms with the same $\hat y_i = (p^c, p^d, p^{r-c}, p^{r-d})$ and get
\ba
\sum\limits_{\substack{0\leq c, d \leq r\\ p^{c-d} \mid M_1, p^{2d-r} \mid M_2}} \mf S_{a,b,M}^{(r)} (c,d) \int_{U(\Z)\bs U(\R)} \varpi(x' \hat y_i y) e\rb{-p^{d-c} M_1 x'_{12} - p^{r-2d} M_2 x'_{24}} dx'.
\ea
Evaluating the integral gives
\begin{align}\label{eq:Hecke_eq2}
\sum\limits_{\substack{0\leq c, d \leq r\\ p^{c-d} \mid M_1, p^{2d-r} \mid M_2}} \mf S_{a,b,M}^{(r)} (c,d) p^{2c+d-\frac{3r}{2}} \frac{A_\varpi(p^{d-c}M_1, p^{r-2d}M_2)}{M^\eta} W_\mu(\iota(M)y). 
\end{align}
Comparing \eqref{eq:Hecke_eq1} and \eqref{eq:Hecke_eq2} gives the result.
\end{proof}

\begin{thm}\label{thm:eigenvalue_Fourier}
Let $\varpi_1 \in V_\pi$ be a cuspidal newform such that $A_{\varpi_1}(1,1) = 1$, and $p\nmid q$ a prime. The Hecke eigenvalues $\lambda(p^r,\pi)$ of $\pi$ with respect to $T(p^r)$ are given by
\ba
\lambda(p,\pi) &= p^{3/2} A_{\varpi_1}(1,p),\\
\lambda(p^r, \pi) &= p^{3r/2}\rb{A_{\varpi_1}(1, p^r) - p^{-1} A_{\varpi_1}(1, p^{r-2})}, & r&\geq 2.
\ea
\end{thm}

\begin{proof}
Plugging in $M = (1,1)$ to \Cref{prp:Hecke_equation} gives
\begin{align}\label{eq:Hecke_equation_11}
\lambda(p^r) A_{\varpi_1}(1,1) = \sum\limits_{0\leq a, b \leq r/2} \lambda_{a,b}^{(r)} (p) A_{\varpi_1}(1,1) = \sum\limits_{0\leq c\leq d\leq r/2} \mf S^{(r)}(c,d) p^{2c+d-\frac{3r}{2}} A_{\varpi_1}(p^{d-c}, p^{r-2d}). 
\end{align}
We evaluate $\mf S^{(r)}(c,d)$ explicitly. We set $A_a := \bp p^c & a\\ & p^d\ep$, and partition the sum
\ba
\mf S^{(r)}(c,d) &= \sum\limits_{0\leq a < p^d} \mf S^{(r)}(c,d; a),\text{ where } \mf S^{(r)}(c,d;a) := \sum\limits_{\substack{\Gamma h_i\sbe S(p^r)\\ A(h_i) = A_a}} \theta(\hat x_i),
\ea
and $A(h_i)$ denotes the top left $2\times 2$ block of $h_i$. Using representatives in \eqref{eq:Sm_coset_rep}, we rewrite
\ba
\mf S^{(r)}(c,d;a) = \sum\limits_{\substack{\Gamma h_i\sbe S(p^r)\\ A(h_i) = A_a}} \theta(\hat x_i) = \sum\limits_{\substack{\Gamma h_i\sbe S(p^r)\\ A(h_i) = A_a}} e\rb{\frac{a}{p^c}+\frac{b_3}{p^{2d}}} = e\rb{\frac{a}{p^c}} \sum\limits_{\substack{\Gamma h_i\sbe S(p^r)\\ A(h_i) = A_a}} e\rb{\frac{b_3}{p^{2d}}}.
\ea
The condition $BD\equiv 0 \pmod{p^r}$ in \eqref{eq:Sm_coset_rep} says 
\begin{align}\label{eq:B_condition}
p^{-r}BD = \bp b_1 p^{-c} - ab_2 p^{-d-c} & b_2 p^{-d}\\ b_2 p^{-c} -ab_3 p^{-d-c} & b_3 p^{-d}\ep \in M_2(\Z). 
\end{align}
Note that the summation over $B$ depends only on $v_p(a)$. We partition the sum with respect to $v_p(a)$. For $v_p(a) \leq c-2$, we have
\ba
\sum\limits_{\substack{0\leq a < p^d\\ v_p(a)\leq c-2}} \mf S^{(r)}(c,d;a) = \sum\limits_{0 \leq v \leq c-2} \sum\limits_{\substack{0< a' < p^{d-v}\\ (a',p)=1}} e\rb{\frac{a'}{p^{c-v}}} \sum\limits_{\substack{\Gamma h_i\sbe S(p^r)\\ A(h_i) = A_{p^v}}} e\rb{\frac{b_3}{p^{2d}}} = 0.
\ea
For $v_p(a) = c-1$, we have $d\geq c \geq 1$, and
\ba
\sum\limits_{\substack{0\leq a < p^d\\ v_p(a)\leq c-1}} \mf S^{(r)}(c,d;a) = \sum\limits_{\substack{0< a' < p^{d-c+1}\\ (a',p)=1}} e\rb{\frac{a'}{p}} \sum\limits_{\substack{\Gamma h_i\sbe S(p^r)\\ A(h_i) = A_{p^{c-1}}}} e\rb{\frac{b_3}{p^{2d}}} = -p^{d-c} \sum\limits_{\substack{\Gamma h_i\sbe S(p^r)\\ A(h_i) = A_{p^{c-1}}}} e\rb{\frac{b_3}{p^{2d}}}.
\ea
The integrality conditions in \eqref{eq:B_condition} forces $p^{d+1} \mid b_3$, $p^d \mid b_2$, and $p^{d+1} \mid b_1 p^{d-c+1} + b_2$. Hence
\ba
\sum\limits_{\substack{0\leq a < p^d\\ v_p(a)\leq c-1}} \mf S^{(r)}(c,d;a) = -p^{d-c} \sum\limits_{\substack{0\leq b_1, b_2, b_3 < p^r\\ p^{d+1} \mid b_3, \; p^d \mid b_2\\ p^{d+1} \mid b_1 p^{d-c+1} + b_2}} e\rb{\frac{b_3}{p^{2d}}} = \begin{cases} -p^{3r-c-2d-1} & \text{ if } d=1,\\ 0 & \text{ otherwise.} \end{cases}
\ea
For $v_p(a) \geq c$, the integrality condition in \eqref{eq:B_condition} forces $p^d \mid b_2, b_3$, and $p^c \mid b_1$. Hence
\ba
\sum\limits_{\substack{0\leq a < p^d\\ v_p(a)\geq c}} \mf S^{(r)}(c,d;a) = p^{d-c} \sum\limits_{\substack{0\leq b_1, b_2, b_3 < p^r\\ p^d \mid b_2, b_3\\ p^c \mid b_1}} e\rb{\frac{b_3}{p^{2d}}}= \begin{cases} p^{3r-c-2d} & \text{ if } d=0,\\ 0 & \text{ otherwise.} \end{cases}
\ea
Hence we conclude 
\ba
\mf S^{(r)} (c,d) = \begin{cases} p^{3r} & \text{ if } (c,d) = (0,0),\\ -p^{3r-4} & \text{ if } (c,d) = (1,1),\\ 0 & \text{ otherwise.}\end{cases}
\ea
Putting this back into \eqref{eq:Hecke_equation_11} gives the statement.
\end{proof}

Hecke eigenvalues can also be expressed in terms of local Satake parameters $\alpha_p, \beta_p$ associated to $\pi$. Without loss of generality, assume $\vb{\alpha_p} \geq \vb{\beta_p} \geq 1$. Then up to some ordering we have $p^{\mu_\pi(p,1)} = \alpha_p$, $p^{\mu_\pi(p,2)} = \beta_p$, and $\sigma_\pi(p) = \mu_\pi(p,1)$. By an identity of Shimura \cite[Theorem 2]{Shimura1963}, we have 
\begin{align}\label{eq:Shimura}
{\small
\sum\limits_{r=0}^\infty \lambda(p^r) x^r = (1-p^2x^2)(1-p^{3/2}\alpha_p x)^{-1}(1-p^{3/2}\alpha_p^{-1} x)^{-1}(1-p^{3/2}\beta_p x)^{-1}(1-p^{3/2}\beta_p^{-1}x)^{-1}.}
\end{align}
For convenience, we define $\sigma_\pi^+(p) = \frac{3}{2}+\sigma_\pi(p)$. 
\begin{lem}\label{lem:eigenvalue_bound}
For a prime $p\nmid q$ and $r\geq 3$ we have
\ba
\max\limits_{0\leq j\leq 3} \vb{\lambda(p^{r-j})} \geq \frac{p^{r\sigma_\pi^+(p)}}{16}.
\ea
\end{lem}
\begin{proof}
We derive from \eqref{eq:Shimura} that
\ba
(1-p^{3/2}\alpha_p^{-1}x)(1-p^{3/2}\beta_px)(1-p^{3/2}\beta_p^{-1} x)\sum\limits_{r=0}^\infty \lambda(p^r) x^r = (1-p^2x^2) \sum\limits_{r=0}^\infty \big(p^{3/2} \alpha_p\big)^r x^r.
\ea
Comparing coefficients gives
\ba
&\lambda(p^r) - \lambda(p^{r-1})p^{3/2}\big(\alpha_p^{-1}+\beta_p+\beta_p^{-1}\big) + \lambda(p^{r-2}) p^3 (\alpha_p^{-1}\beta_p+\alpha_p^{-1}\beta_p^{-1}+1) + \lambda(p^{r-3}) p^{9/2} \alpha_p^{-1}\\
= &p^{3r/2} (\alpha_p^r - p^{-1} \alpha_p^{r-2}). 
\ea
Assume the contrary. Then the left hand side is bounded by
\ba
\frac{p^{r\sigma_\pi^+(p)}}{2} \leq p^{r\sigma_\pi^+(p)} - p^{2+(r-2)\sigma_\pi^+(p)} \leq p^{3r/2} \vb{\alpha_p^r - p^{-1} \alpha_p^{r-2}},
\ea
a contradiction.
\end{proof}

\begin{lem}\label{lem:Fourier_bound}
Let $\varpi_1 \in V_\pi$ be a cuspidal newform such that $A_{\varpi_1}(1,1) = 1$, $p\nmid q$ a prime, and $r_0\in \N_0$. Then the inequality
\ba
\vb{A_{\varpi_1}(1,p^r)} \geq \frac{p^{r\sigma_\pi(p)}}{32}
\ea
holds for some $r_0 \leq r \leq r_0+5$. 
\end{lem}
\begin{proof}
By \Cref{lem:eigenvalue_bound}, we have
\ba
\vb{\lambda(p^r)} \geq \frac{p^{r\sigma_\pi^+(p)}}{16}
\ea
for some $r_0+2 \leq r \leq r_0+5$. By \Cref{thm:eigenvalue_Fourier}, we have
\ba
p^{3r/2}\rb{\vb{A_{\varpi_1}(1, p^r)} + p^{-1} \vb{A_{\varpi_1}(1, p^{r-2})}} \geq \frac{p^{r\sigma_\pi^+(p)}}{16},
\ea
and the statement follows.
\end{proof}

\section{Poincar\'e series and the Kuznetsov formula}

Let $E: \R_+^2 \to \C$ be a fixed function with compact support, $X \in \R_+^2$ a ``parameter''. We define
\ba
E^{(X)} (y_1, y_2) = E(X_1y_1, X_2y_2), 
\ea
and a right $K$-invariant function $F^{(X)}: \Sp(4,\R) \to \C$ by 
\begin{align}\label{eq:F_def}
F^{(X)} (xy) = \theta(x) E^{(X)}(\y(y))
\end{align}
for $x\in U(\R)$ and $y\in T(\R_+)$, where $\theta = \theta_{(1,1)}$ is as in \eqref{eq:character_def}. For $N\in\N^2$, we define the Poincar\'e series of level $q$ to be
\ba
P_N^{(X)} (xy) = \sum\limits_{\gamma\in U(\Z) \bs \Gamma_0(q)} F^{(X)} (\iota(N)\gamma xy). 
\ea
Note that $F^{(X)}(\iota (N)xy) = \theta_N(x) E^{(X)} (N \y(y)) = \theta_N(x) E(XN\y(y))$. For $w\in W$, let $G_w = UwTU$, and $\Gamma_w := U(\Z) \cap w^{-1} U(\Z)^\top w$. Let $R_w(q)$ be a complete system of coset representatives for $P_0\cap \Gamma_0(q) \bs \Gamma_0(q) \cap G_w/ \Gamma_w$. 

We compute the Fourier coefficients of the Poincar\'e series:
\ba
&\int_{U(\Z)\bs U(\R)} P_M^{(X)}(x,y) \ol{\theta_N(x)} dx\\
= &\sum\limits_{\gamma \in P_0\cap \Gamma_0(q)\bs \Gamma_0(q)} \int_{U(\Z)\bs U(\R)} F^{(X)}(\iota(M)\gamma xy) \ol{\theta_N(x)} dx\\
= & \sum\limits_{w\in W} \sum\limits_{\gamma\in R_w(q)} \sum\limits_{\ell\in\Gamma_w} \int_{U(\Z)\bs U(\R)} F^{(X)} (\iota(M)\gamma \ell xy) \ol{\theta_N(x)} dx\\
= & \sum\limits_{w\in W} \sum\limits_{c\in \N^2} \Kl_{q,w}(c,M,N) \int_{U_w(\R)} F^{(X)} (\iota(M)c^*wxy) \ol{\theta_N(x)} dx.
\ea
For fixed $y$, it follows from \Cref{lem:1} and $E$ having compact support that the $c$-sum runs over a finite set, and $U_w(\R)$ runs over a compact domain. In particular, the right hand side is absolutely convergent. 

Now let $\varpi$ be an automorphic form, not necessarily cuspidal, in the spectrum of $L^2(\Gamma_0(q)\bs \Sp(4,\R)/K)$. By unfolding, \eqref{eq:Fourier_in_Whittaker} and a change of variables $\iota(N)y \mapsto y$, we obtain
\ba
\pb{\varpi, P_N^{(X)}} = \int_{T(\R_+)} \int_{U(\Z)\bs U(\R)} \varpi(xy) \theta_N(-x) \ol{E^{(X)}(N\cdot \y(y))} dx d^*y = N^\eta A_\varpi (N) \pb{W_\mu, E^{(X)}}.
\ea
By Parseval, we obtain
\ba
\pb{P_M^{(X)}, P_N^{(X)}} = M^\eta N^\eta \int_{(q)} \ol{A_\varpi(M)} A_\varpi(N) \vb{\pb{W_\mu, E^{(X)}}}^2 d\varpi. 
\ea
Meanwhile, unfolding the inner product directly gives
\ba
\pb{P_M^{(X)}, P_N^{(X)}} &= \int_{T(\R_+)} \int_{U(\Z)\bs U(\R)} P_M^{(X)} \theta_N(-x) \ol{E^{(X)}(N\cdot \y(y))} dx d^*y\\
&= \sum\limits_{w\in W} \sum\limits_{c\in\N^2} \Kl_{q,w}(c,M,N) \int_{T(\R^+)} \int_{U_w(\R)} F^{(X)}(\iota(M) c^* wxy) \theta_N(-x) \ol{E(XN \cdot \y(y))} dx d^*y.
\ea
Now define
\begin{align}\label{eq:A_def}
A = \iota(XM) c^* w \iota(XN)^{-1} w^{-1} = \iota((XM) \cdot \tensor[^w]{(XN)}{}) c^* \in T(\R_+).
\end{align}
Then $\y(A)^\eta c_1 c_2 = \rb{ (XM) \cdot \tensor[^w]{(XN)}{}}^\eta$. By change of variables $\iota(XN)y \mapsto y$, $\iota(XN) x \iota(XN)^{-1} \mapsto x$, we can express $\pb{P_M^{(X)}, P_N^{(X)}}$ as
\ba
\sum\limits_{w\in W} \sum\limits_{c\in \N^2} \Kl_{q,w}(c,M,N) \frac{(XM)^\eta (XN)^\eta}{c_1c_2 \y(A)^\eta} \int_{T(\R_+)} \int_{U_w(\R)} F^{(X)}(\iota(X)^{-1} A wxy) \theta_{X^{-1}}(-x) \ol{E(\y(y))} dx d^*y.
\ea
We then conclude a Kuznetsov-type trace formula.
\begin{lem}\label{lem:Kuznetsov}
Let $M,N \in \N^2$, $X\in \R_+^2$, $E$ a function on $\R_+^2$ with compact support, and define $F^{(X)}$ as in \eqref{eq:F_def}. Then
\begin{align}
&\int_{(q)} \ol{A_\varpi(M)} A_\varpi(N) \vb{\pb{W_\mu, E^{(X)}}}^2 d\varpi\\
= &\sum\limits_{w\in W} \sum\limits_{c\in \N^2} \Kl_{q,w}(c,M,N) \frac{X^{2\eta}}{c_1c_2 \y(A)^\eta} \int_{T(\R_+)} \int_{U_w(\R)} F^{(X)}(\iota(X)^{-1} A wxy) \theta_{X^{-1}}(-x) \ol{E(\y(y))} dx d^*y, \nonumber
\end{align}
with $A$ as in \eqref{eq:A_def}. 
\end{lem}

\section{Proof of theorems}\label{section:proofs}

We establish the following proposition, from which the other theorems are proved.

\begin{prp}\label{prp:AZ}
Keep the notations as above. Let $m\in\N$ be coprime to $q$ and $Z\geq 1$. Then
\ba
\int_{(q)} \vb{A_\varpi(1,m)}^2 Z^{2\sigma_\pi(\infty)} \delta_{\lambda_\varpi \in I} d\varpi \ll_{I,\varepsilon} q^\varepsilon
\ea
uniformly in $mZ\ll q^2$ for a sufficiently small implied constant depending on $I$.
\end{prp}

\begin{proof}
We take $X=(1,Z)$, $M=N=(1,m)$, and apply \Cref{lem:Kuznetsov}. By \Cref{lem:5}, there is a finite set of compactly supported functions $E_j$ such that
\begin{align}\label{eq:WEj_bound}
Z^{2\eta_2+2\sigma_\pi(\infty)} \delta_{\lambda_\varpi\in I} \ll_I \sum\limits_j \vb{\pb{W_\mu, E_j^{(X)}}}^2.
\end{align}
Now we consider the arithmetic side of the Kuznetsov formula for a fixed $E^{(X)} = E_j^{(X)}$. It suffices to consider the Weyl elements $w\in W$ for which the Kloosterman sum $\Kl_{q,w}(c,M,N)$ does not vanish, namely, $w=\id, s_\alpha s_\beta s_\alpha, s_\beta s_\alpha s_\beta, w_0$. 

For $w=\id$, we have $c_1 = c_2 = 1$, and hence the contribution is $O(Z^{2\eta_2}) = \mc O(Z^3)$. 

Now let $w\in\cb{s_\alpha s_\beta s_\alpha, s_\beta s_\alpha s_\beta, w_0}$. Apply \Cref{lem:1} with $B = (XM)\cdot \tensor[^w]{(XN)}{}$. Concretely, we set
\ba
(B_1, B_2) = \begin{cases} (mZ,1) & \text{ if } w = s_\alpha s_\beta s_\alpha,\\ (1, (mZ)^2) & \text{ if } w = s_\beta s_\alpha s_\beta, w_0. \end{cases}
\ea
Then we obtain
\ba
c_1 \ll B_1 B_2^{1/2} &= mZ, & c_2 \ll B_1B_2 = \begin{cases} mZ & \text{ if } w=s_\alpha s_\beta s_\alpha,\\ (mZ)^2 & \text{ if } w=s_\beta s_\alpha s_\beta, w_0.\end{cases}
\ea
We assume $mZ\ll q^2$ with a sufficiently small implied constant, such that 
\begin{align}\label{eq:c_condition}
c_1, c_2 < q^2 \text{ for } &w=s_\alpha s_\beta s_\alpha, & \text{ and } c_1< q^2, c_2< q^4 \text{ for } &w=s_\beta s_\alpha s_\beta, w_0
\end{align}
always hold. Now we consider the Kloosterman sums
\ba
\Kl_{q,w}(c, M, N) = \sum\limits_{xc^*wx'\in U(\Z) \bs G_w(\Q)\cap \Gamma_0(q) / U_w(\Z)} \theta_M(x) \theta_N(x'),
\ea
where the entries of $M = (M_1, M_2)$ and $N = (N_1, N_2)$ are coprime to $q$. The Kloosterman sums are nonzero only when \eqref{eq:Kl_nonempty} and \eqref{eq:Kl_q_defined} are satisfied. By \eqref{eq:Kl_mult}, \eqref{eq:Kl_cprime_bound} and \eqref{eq:c_condition}, the problem reduces to computing the Kloosterman sums 
\ba
&\Kl_{q, s_\alpha s_\beta s_\alpha} \rb{(q,q), M, N}, & &\Kl_{q, s_\beta s_\alpha s_\beta} \rb{(q, q^2), M, N},\\
&\Kl_{q, w_0}\rb{(q,q^2), M, N}, & &\Kl_{q, w_0}\rb{(q,q^3), M, N}.
\ea
From \Cref{section:Kl_eval}, we see that only $\Kl_{q,s_\beta s_\alpha s_\beta}\rb{(q,q^2),M,N}$ does not vanish. So only $w=s_\beta s_\alpha s_\beta$ contributes.

The next step is to estimate for $w=s_\beta s_\alpha s_\beta$ the integral
\ba
&\vb{\int_{T(\R_+)}\int_{U_w(\R)} F^{(X)} (\iota(X)^{-1} Awxy) \theta_{X^{-1}}(-x) \ol{E(\y(y))} dx d^*y}\\
\leq &\int_{T(\R_+)} \int_{U_w(\R)} \vb{E(\y(Awxy)) E(\y(y))}.
\ea
This integral is bounded by the size of the set of $x\in U_w(\R)$ such that $\y(Awxy)$ lies in the support of $E$. Using \Cref{lem:1} and \Cref{lem:3}, we deduce that
\ba
&\vb{\int_{T(\R_+)}\int_{U_w(\R)} F^{(X)} (\iota(X)^{-1} Awxy) \theta_{X^{-1}}(-x) \ol{E(\y(y))} dx d^*y}\\
\ll_E &\vol\cbm{x\in U_w(\R)}{\Delta_1(wx)\ll_E \y(A)_1\y(A)_2^{1/2}, \; \Delta_2(wx)\ll_E \y(A)_1\y(A)_2} \ll_E \y(A)^{\eta(1+\varepsilon)}.
\ea
So the contribution from $w=s_\beta s_\alpha s_\beta$ is given by
\ba
&\sum\limits_{c\in \N^2} \Kl_{q,w}(c,M,N) \frac{X^{2\eta}}{c_1c_2 \y(A)^\eta} \int_{T(\R_+)} \int_{U_w(\R)} F^{(X)}(\iota(X)^{-1} A wxy) \theta_{X^{-1}}(-x) \ol{E(\y(y))} dx d^*y\\
\ll_E &\sum\limits_{c'_1\ll mZ/q} \frac{Z^{2\eta_2} \y(A)^{\varepsilon}}{q} \ll Z^3 q^\varepsilon.
\ea
Combining the estimates with \eqref{eq:WEj_bound}, we obtain
\ba
\int_{(q)} \vb{A_\varpi(1,m)}^2 Z^{3+2\sigma_\pi(\infty)} \delta_{\lambda_\varpi\in I} d\varpi \ll_I \int_{(q)} \vb{A_\varpi(1,m)}^2 \vb{\pb{W_\mu, E^{(X)}}}^2 d\varpi \ll_{\varepsilon} Z^3 q^\varepsilon. 
\ea
Dividing both sides by $Z^3$ yields the theorem. 
\end{proof}

\begin{proof}[Proof of \Cref{thm:lambdaZ}]
It follows easily from \Cref{prp:AZ}, \Cref{thm:eigenvalue_Fourier} and the estimate \eqref{eq:A11_estimate} that
\ba
\sum\limits_{\pi\in\mc F_I(q)} \vb{\lambda'(m,\pi)}^2 Z^{2\sigma_\pi(\infty)} \ll_{\varepsilon} q^{3+\varepsilon} \int_{(q)} \vb{A_\varpi(1,m)}^2 Z^{2\sigma_\pi(\infty)} \delta_{\lambda_\varpi \in I} \ll_{I,\varepsilon} q^{3+\varepsilon}.
\ea
\end{proof}

\begin{proof}[Proof of \Cref{thm:alpham}]
This is just a simple variation of the proofs above. Again we have
\ba
\sum\limits_{\pi \in\mc F_I(q)} \Big|\sum\limits_{\substack{m\leq x\\ (m,q)=1}} \alpha(m) \lambda'(m,\pi)\Big|^2 \ll_{\varepsilon} q^{3+\varepsilon} \int_{(q)} \Big|\sum\limits_{\substack{m\leq x\\ (m,q)=1}} \alpha(m) A_\varpi(M)\Big|^2 \delta_{\lambda_\varpi\in I} d\varpi
\ea
\ba
= q^{3+\varepsilon} \sum\limits_{\substack{m_1, m_2\leq x\\ (m_1m_2, q) = 1}} \alpha(m_1) \ol{\alpha(m_2)} \int_{(q)} A_\varpi(M_1) \ol{A_\varpi(M_2)} \delta_{\lambda_\varpi \in I} d\varpi,
\ea
where $M = (1,m)$, $M_1 = (1, m_1)$, $M_2 = (1,m_2)$. Now we apply \Cref{lem:Kuznetsov} and evaluate the Kloosterman sums on the arithmetic side. For $w\neq\id$, apply \Cref{lem:1} with $B = M_1 \cdot \tensor[^w]{M}{_2}$. We get
\ba
c_1&\ll (m_1m_2)^{1/2}\leq x, & c_2&\ll m_1\leq x & \text{ for } w&=s_\alpha s_\beta s_\alpha,\\
c_1&\ll (m_1m_2)^{1/2}\leq x, & c_2&\ll m_1m_2\leq x^2 & \text{ for } w&= s_\beta s_\alpha s_\beta, w_0.
\ea
Note that when $x\ll q$ with a sufficiently small implied constant, the condition $(m,q)=1$ is void, and we deduce from \eqref{eq:Kl_nonempty} that the Kloosterman sums $\Kl_{q,w}(c,M,N)$ are empty for $w\neq \id$. Hence only the trivial Weyl element contributes, and we obtain the desired bound.
\end{proof}

\begin{proof}[Proof of \Cref{cor:Lvalue}]
Observe that for $\pi \in \mc F_I(q)$ an approximate functional equation has length $q^{1/2}$ (see \cite[Section 5]{IK2004}). So, for all but $O(1)$ cuspidal representations $\pi \in\mc F_I(q)$ (and $\varepsilon<1/2$) we have
\ba
\vb{L(1/2+it, \pi)}^2 \ll_{I,t,\varepsilon} q^{\varepsilon} \sum\limits_{2^j = M\leq q^{1/2+\varepsilon}} \frac{1}{M} \Big| \sum\limits_{M\leq m\leq 2M} \lambda'_\pi(m) \Big|^2.
\ea
The statement then follows from \Cref{thm:alpham}.
\end{proof}

\begin{proof}[Proof of \Cref{thm:density_thm}]
We first assume $v=p\neq q$ is a finite place. We choose $\nu_0$ maximal such that $p^{\nu_0} \ll q^2$ with an implied constant that is admissible to \Cref{prp:AZ}. Then by \Cref{lem:Fourier_bound} and the estimate \eqref{eq:A11_estimate}, there exists $\nu_0-5 \leq \nu_\pi \leq \nu_0$ such that
\ba
\vb{A_\varpi(1, p^{\nu_\pi})}^2 \gg q^{-3-\varepsilon} p^{2\nu_\pi \sigma_\pi(p)}. 
\ea
Note that $p^{\nu_\pi} \asymp q^2$. We apply \Cref{prp:AZ} with $m=p^{\nu_\pi}$, $Z=1$, and conclude that
\ba
N_p(\sigma, \mc F_I(q)) \leq \sum\limits_{\pi\in \mc F_I(q)} \frac{p^{2\nu_\pi\sigma_\pi(p)}}{p^{2\nu_\pi\sigma}} \ll q^{3-4\sigma+\varepsilon} \int_{(q)} \sum\limits_{\nu_0-5\leq \nu \leq \nu_0}\vb{A_\varpi(1,p^\nu)}^2 \delta_{\lambda_{\varpi\in I}} \ll_{I,\varepsilon} q^{3-4\sigma+\varepsilon}. 
\ea

For $v=\infty$, we use the estimate \eqref{eq:A11_estimate}, apply \Cref{prp:AZ} with $m=1$, $Z\ll q^2$, and conclude that
\ba
N_\infty(\sigma, \mc F_I(q)) \leq \sum\limits_{\pi \in \mc F_I(q)} Z^{2\sigma_\pi(\infty)-2\sigma} \ll q^{3-4\sigma+\varepsilon} \int_{(q)} \vb{A_\varpi(1,1)}^2 Z^{2\sigma_\pi(\infty)} \ll_{I,\varepsilon} q^{3-4\sigma+\varepsilon}.
\ea
\end{proof}

\section{Appendix: Computation of Fourier coefficients}

In this appendix, we outline an algorithm for computing arbitrary Fourier coefficients of a cuspidal newform $\varpi_1 \in V_\pi$ with $A_{\varpi_1}(1,1) = 1$. For this purpose, it suffices to compute the actions of $T(p)$ and $T_{0,1}^{(2)}(p)$, which generate the Hecke algebra. By \Cref{prp:Hecke_equation}, we compute 
{\small
\begin{align}\label{eq:Tp_formula}
\lambda(p,\pi) A_{\varpi_1}(M_1, M_2) = p^{3/2}\big(A_{\varpi_1}(M_1, pM_2)\underbrace{+A_{\varpi_1}(p^{-1}M_1,pM_2)}_{\text{if } p\mid M_1} \underbrace{+A_{\varpi_1}(pM_1, p^{-1}M_2) + A_{\varpi_1}(M_1, p^{-1}M_2)}_{\text{if } p\mid M_2}\big),
\end{align}
}
and if $p\nmid M_2$,
{\small
\begin{align}\label{eq:T012p_formula}
\big(\lambda_{0,1}^{(2)}(p,\pi)+1\big) A_{\varpi_1}(M_1, M_2) = p^2\big(A_{\varpi_1}(pM_1, M_2)\underbrace{+A_{\varpi_1}(p^{-1}M_1, p^2M_2) + A_{\varpi_1}(p^{-1}M_1, M_2)}_{\text{if } p\mid M_1}\big).
\end{align}
}

We proceed to show how the Fourier coefficients $A_{\varpi_1}(p^{k_1}, p^{k_2})$ are obtained. Starting from $A_{\varpi_1}(1,1) = 1$, we apply \eqref{eq:Tp_formula} and \eqref{eq:T012p_formula} with $M=(1,1)$ and solve the coefficients
\ba
A_{\varpi_1}(p,1) &= p^{-2} \rb{\lambda_{0,1}^{(2)}(p)+1}, & A_{\varpi_1}(1,p) = p^{-3/2} \lambda(p).
\ea
Inductively, suppose the Fourier coefficients $A_{\varpi_1}(p^{k_1}, p^{k_2})$ are known for all $k_1+k_2 \leq r$. For $0\leq k\leq r$, applying \eqref{eq:Tp_formula} with $M=(p^k, p^{r-k})$ yields the coefficient $A_{\varpi_1}(p^k, p^{r-k+1})$. Then, applying \eqref{eq:T012p_formula} with $M=(p^r, 1)$ yields the coefficient $A_{\varpi_1}(p^{k+1},1)$, since the coefficient $A_{\varpi_1}(p^{k-1}, p^2)$ has already been determined. This shows that the Fourier coefficients $A_{\varpi_1}(p^{k_1}, p^{k_2})$ with $k_1+k_2\leq r+1$ can be expressed in terms of $\lambda(p)$ and $\lambda_{0,1}^{(2)}(p)$, finishing the induction. 

Writing $X := p^{-3/2} \lambda(p,\pi)$ and $Y := p^{-2}\rb{\lambda_{0,1}^{(2)}(p,\pi) + 1}$, the Fourier coefficients $A_{\varpi_1}(p^{k_1}, p^{k_2})$ for small $k_i$ are computed in the following table:
{\tiny
\ba
\begin{array}{c | cccccc}
A_{\varpi_1}(p^{k_1},p^{k_2}) & k_2 = 0 & k_2 = 1 & k_2=2 & k_2=3\\
\hline\\
k_1=0 & 1 & X & X^2-Y-1 & X^3-2XY-X \\
\\
k_1=1 & Y & XY-X & \begin{array}{c}X^2Y-X^2\\-Y^2-Y+1\end{array} & \begin{array}{c}X^3Y-X^3\\-2XY^2+2X\end{array} \\
\\
k_1=2 & -X^2+Y^2+Y & \begin{array}{c}-2X^2Y+Y^3\\+X^2+2Y^2-1\end{array} & \begin{array}{c}-X^4+X^2Y^2+X^2Y\\-Y^3+2X^2-2Y^2\end{array} & \begin{array}{c}-X^5+X^3Y^2+2X^3Y\\-2XY^3+2X^3-2XY^2-X\end{array}\\
\\ 
k_1=3 & \begin{array}{c}-2X^2Y+Y^3+X^2\\+2Y^2-1\end{array} & \begin{array}{c}-2X^3Y+XY^3+2X^3\\+XY^2-2X\end{array} & \begin{array}{c}-2X^4Y+X^2Y^3+2X^4\\+3X^2Y^2-Y^4+X^2Y-3Y^3\\-4X^2-Y^2+2Y+1\end{array} & \begin{array}{c}-2X^5Y+X^3Y^3+2X^5\\+5X^3Y^2-2XY^4-X^3Y\\-4XY^3-5X^3+4XY+2X\end{array}
\end{array}
\ea
}
From \Cref{thm:eigenvalue_Fourier}, we obtain $\lambda(p^2, \pi) = p^3(X^2-Y-1) - p^2$. Hence the Fourier coefficients can also be expressed in terms of eigenvalues $\lambda(p^r,\pi)$ of standard Hecke operators.

It is evident from the \Cref{prp:Hecke_equation} that Fourier coefficients are multiplicative, that is,
\begin{align}\label{eq:Fourier_multiplicative}
A_{\varpi_1}(M_1N_1, M_2N_2) = A_{\varpi_1}(M_1,M_2) A_{\varpi_1}(N_1,N_2) \text{ if } (M_1M_2, N_1N_2) = 1.
\end{align}
Using \eqref{eq:Fourier_multiplicative}, and \eqref{eq:Teps_action} for negative coefficients, we are able to compute $A_{\varpi_1}(M)$ for every $M\in\Z^2$.


\end{document}